%% file: Manuscript.tex
\documentclass[final,hidelinks,onefignum,onetabnum]{siamart220329}

\input{./packages}
\input{./commands}

\usepackage{soul}

\title {Graph $p$-Laplacian eigenpairs as saddle points of a family of spectral energy functions\thanks{\funding{Piero Deidda and Mario Putti are members of the Gruppo Nazionale Calcolo Scientifico-Istituto Nazionale di Alta Matematica (GNCS-INdAM). Piero Deidda was supported by the MUR-PRO3 grant STANDS. Mario Putti was partially funded by the European Union-Next Generation EU National Recovery and Resilience Plan [NRRP], Mission 4,
Component 2, Investment 1.5 - Call for tender No. 3277 of 30 dicembre 2021 ECS00000043, No. 1058 of June 23, 2023, CUP C43C22000340006,
“iNEST: Interconnected Nord-Est Innovation Ecosystem”.}}}

\author{Piero Deidda\thanks{Scuola Normale Superiore, Pisa, Italy - Gran Sasso Science Institute, L'Aquila, Italy (\email{piero.deidda@sns.it}).}
\and Nicola Segala\thanks{Department of Mathematics, “Tullio Levi-Civita", University of Padova, Italy\textsuperscript{1}.\newline
\textsuperscript{1}Now at Almaviva Bluebit SpA (\email{N.segala@almaviva.it})}
\and  Mario Putti\thanks{Department of Agronomy, Food, Natural Resources, Animals and Environment, University of Padova, Italy} (\email{mario.putti@unipd.it})}

\begin{document}

\maketitle

\begin{abstract}
 We address the problem of computing the graph $p$-Laplacian eigenpairs for $p\in (2,\infty)$. 
 We propose a reformulation of the graph $p$-Laplacian eigenvalue problem in terms of a constrained weighted Laplacian eigenvalue problem and discuss theoretical and computational advantages.
 We provide a correspondence between $p$-Laplacian eigenpairs and linear eigenpair of a constrained generalized weighted Laplacian eigenvalue problem. As a result, we can assign an index to any $p$-Laplacian eigenpair that matches the Morse index of the $p$-Rayleigh quotient evaluated at the eigenfunction. In the second part of the paper we introduce a class of spectral energy functions that depend on edge and node weights. We prove that differentiable saddle points of the $k$-th energy function correspond to $p$-Laplacian eigenpairs having index equal to $k$. Moreover, the first energy function is proved to possess a unique saddle point which corresponds to the unique first $p$-Laplacian eigenpair. Finally we develop novel gradient-based numerical methods suited to compute $p$-Laplacian eigenpairs for any $p\in(2,\infty)$ and present some experiments. 
\end{abstract}



\section{Introduction}

The $p$-Laplace operator arises as a natural generalization of the Laplace-Beltrami operator in variational problems involving the $p$-norm of the gradient of an objective function $f$, $\|\grad f\|_p$.
Its numerous applications make it one of a deeply studied nonlinear operators both in the continuous and in the discrete settings \cite{Elmoataz1, Elmoataz2, Lind, Bhuler}.
In this paper we focus our study on the spectrum of the $p$-Laplace operator defined on graphs and in particular to the computation of the $p$-Laplacian eigenpairs. The eigenpairs of a $p$-Laplacian are typically defined as the critical points/values of the family of Rayleigh quotients given by
\begin{equation*}
  \rayl_p(f)=\frac{\|\grad f\|_p^p}{\|f\|^p}\,,
\end{equation*}
where $f$ is a function defined on the graph nodes and different norms at the denominator can be considered \cite{franzina2010existence}.
The interest for nonlinear eigenpairs is varied and includes data filtering, clustering, and partitioning, with linked applications in the field of optimal transportation problems~\cite{burger2016spectral, bungert2022eigenvalue, gilboa2018nonlinear}.

A remarkable application of the $p$-Laplacian spectrum can be found in data clustering and partitioning. 
Indeed, different authors have addressed this problem in both the discrete ~\cite{chang2016spectrum,Tudisco1,Bhuler,hein2010inverse,ZhangNodalDO, Hua, deidda2023PhdThesis, deidda2024_inf_eigenproblem}  and the continuous settings~\cite{parini2010second,Kawohl2003, Lind2, Lind3, Bungert1, bungert2022eigenvalue, Esposito}. It has been proved that the variational eigenvalues of the  $1$-Laplacian and $\infty$-Laplacian, and more generally the limit of the variational eigenvalues of the $p$-Laplacian as $p$ goes to $1$ and $\infty$, provide good approximations of the Cheeger constants and spherical packing radii of the domain. 
In particular for $p<2$ such approximations improve the known relationships between the Cheeger constants and the Laplacian eigenvalues already observed by Cheeger himself~\cite{Cheeger}.
We recall that Cheeger constants are used to  quantify the number of clusters in the domain while the spherical packing radii establish the maximal reciprocal distance between any set of a given number of nodes of the graph.

In general, $p$-Laplacian eigenpairs, both for $p$ close to one and to infinity, find applications in data classification tasks. The $1$-eigenpairs are related to clustering tasks with any eigenvalue being equal to the isoperimetric constant of some subgraph \cite{chang2016spectrum, ZhangNodalDO}. Analogously infinity-eigenpairs are related to partitioning problems based on the distance between nodes, see e.g. the independence numbers of a graph \cite{ABIAD20192875, FIRBY199727, MeirMoon}, since any infinity eigenvalue corresponds to the distance between some nodes of the graph \cite{deidda2024_inf_eigenproblem}. We refer to \cite{saito23effective_p_resistance} for analogous applications of $p$-resistance on graphs, which is also related to graph cuts for $p$ that goes to one and graph distances for $p$ that goes to infinity.

Laplacian eigenpairs have been largely employed in classification and clustering tasks in classical unsupervised learning methods, see \cite{von2007tutorial} and references therein, and  more recently in semisupervised methods based on graph neural networks \cite{dwivedi2023benchmarking}. Replacement of the Laplacian spectrum by $p$-Laplacian eigenpairs has shown improvements in the performance of clustering methods in \cite{Bhuler}. However, since the computation of higher eigenpairs is a difficult task and presents several open problems, so far these applications have been restricted to binary classification tasks, where only the first nonzero eigenpair is used.
%


%
The open problems in the study of the $p$-Laplacian spectrum regard both theoretical and computational points of view.
In fact, a number of properties of the linear ($p=2$) Laplacian eigenfunctions are lost in the nonlinear ($p\ne2$) case.  The first and probably most notable difficulty is consequential to the fact that the cardinality of the $p$-Laplacian spectrum is not known and can exceed the dimension of the space \cite{zhang2021homological,DEIDDA2023_nod_dom,Amghibech1}. This clearly yields the loss of the notion of multiplicity of an eigenvalue and of  orthogonality of the eigenfunctions.
The introduction of the variational eigenpairs allows to partially overcome these difficulties.
Variational eigenpairs are defined by a $\min\max$ theorem that generalizes the classical Rayleigh-Ritz characterization of the eigenvalues of a symmetric matrix. 
As a consequence, the cardinality of the variational eigenvalues is always equal to the dimension of the space.
Hence, the variational eigenvalues provide a partition of the $p$-Laplacian spectrum in non-empty subintervals. The position of a general eigenvalue in one of these intervals has some nontrivial implications as it affects the characteristic ``frequency'' of the corresponding eigenfunction~\cite{DEIDDA2023_nod_dom,Tudisco1}\,.
Moreover, it is possible to define a notion of multiplicity for the variational eigenvalues that is consistent with the notion of multiplicity in the linear case $p=2$ \cite{struwe,DEIDDA2023_nod_dom,ZhangNodalDO}.  

Clearly, the numerical approximation of the $p$-Laplacian eigenpairs presents the same difficulties in addition to the natural issues arising in all discretization processes. Among these, we have identified two fundamental tasks that are not or only partially addressed in the literature:
\begin{enumerate*}[i)]
\item develop consistent numerical algorithms, i.e., algorithms for which convergence toward solutions of the eigenequation can be proved; 
\item classification of the approximated eigenpairs in terms of the variational spectrum. 
\end{enumerate*}
Given the above mentioned difficulties and uncertainties, a scheme that is consistent in the above sense is difficult to provide. In particular, to the best of our knowledge, no methods exist to identify the variational eigenvalues within a set of eigenpairs.

Notwithstanding the above difficulties and driven by the continuously excalating interest in data science, different algorithms for the numerical solution of the $p$ eigenproblem have been proposed in the last few years~\cite{Yao2007NumericalMF, hein2010inverse, Bungert4, Bozorgnia_first_p-Lap_eigen, BUNGERTgradient_flows_power_method}.
In~\cite{Yao2007NumericalMF}, the authors develop a scheme capable of computing a sequence of $N$ eigenpairs based on a local min max method. The functionality of the approach is based on the fact that it is possible to characterize a $p$-Laplacian eigenvalue as a local maximum of the $p$-Rayleigh quotient on some linear subspace spanned by eigenvectors corresponding to smaller frequencies. However, to the best of our knowledge, there is no theoretical evidence for the existence of a sequence of eigenvalues satisfying such properties.

The situation improves when looking for extremal eigenpairs. In fact, for the nonlinear power method and the gradient flow scheme proposed in \cite{hein2010inverse, Bozorgnia_first_p-Lap_eigen, Bungert4, BUNGERTgradient_flows_power_method} to compute the extremal eigenpairs, it is possible to prove  convergence. However, for eigenpairs that are not ground states (i.e., lowest value/points of the $p$-Rayleigh quotient), no a-priori information is available about the position in the spectrum of the approximated eigenpair. In addition, none of these methods is suited to compute a full sequence of eigenpairs.

In this work, we try to address some of the above-mentioned fundamental tasks by providing new insights and results on the $p$-Laplacian eigenvalue problem on graphs. These results lead naturally to an original numerical scheme that overcomes some of the limitations identified above. The foundation of our work is a novel re-interpretation of the graph $p$-Laplacian eigenvalue problem as a constrained linear weighted Laplacian eigenproblem. The consequences of this reformulation are manyfold.
First, it becomes possible to naturally assign a linear index to every $p$-Laplacian eigenvalue $\eigenval$ by simply assigning to it the corresponding index of the associated linear eigenvalue problem.
Second, we are able to prove that, for any eigenpair $(\eigenval,\eigenfunction)$, this linear index of $\eigenval$ matches the Morse index of the $p$-Rayleigh quotient functional in $f$, thus providing additional information about the behaviour of the $p$-Rayleigh quotient in a neighborhood of $f$.
%
%

Based on this reformulation and inspired by the Dynamical-Monge-Kantorovich method introduced in \cite{facca2018towards, facca3, facca2019transport, facca2}, we consider the case of $p\in(2,\infty)$ and characterize the $p$-Laplacian eigenpairs as critical points of a family of energy functions defined on the domains of the node and edge weights. 
Such energy functions are indexed from $1$ to $N$, where $N$ is the dimension of the graph, and thus provide a natural indexing for the eigenpair approximations.
Indeed, we are able to prove that the unique saddle point of the $1$-st energy function corresponds to the unique first $p$-Laplacian eigenpair. 
Moreover, we prove that any differentiable saddle point of the $k$-th energy function corresponds to a $p$-Laplacian eigenpair having linear index equal to $k$. 
We then derive gradient flows for our energy functions and develop numerical algorithms for the computation of $p$-Laplacian eigenpairs. From a numerical point of view, our methods compute $p$-Laplacian eigenpairs as limits of sequences of linear eigenvalue problems, and we can then exploit the vast literature available for this last problem.
Note that we are able to compute higher $p$-Laplacian eigenpairs without any prior information about the lower ones. Indeed, the choice of the index of the energy function prescribes a-priori the type of saddle point we converge to.
Lastly, considering again the first energy function, since we know that its unique saddle point corresponds to the unique $1$-st $p$-Laplacian eigenpair, we can conclude that our method converges exactly to that eigenpair. 

We point out that the energy functions here introduced are well defined also in the $p=\infty$ case. This leads us to conjecture the validity of our results also in the case $p=\infty$.
However, the theoretical results that we prove in the case $p\in(2,\infty)$ cannot be extended in a straightforward manner to the case $p=\infty$, which will be the subject for a future paper.
We wish to conclude by observing that some very recent duality results \cite{zhang2021discrete,tudisco2022nonlinear,bungert2022eigenvalue} relate the $p$-Laplacian eigenvalue problem on the nodes of the graph to the $q$-Laplacian eigenvalue problem on the edges of the graph, where $p$ and $q$ are H\"older conjugates.
In particular, in~\cite{zhang2021discrete,tudisco2022nonlinear} the authors prove that there is a $1$-to-$1$ correspondence between the non-zero eigenpairs of the node $p$-Laplacian and the edge $q$-Laplacian. 
Thus, extending our results to the edge $q$-Laplacian for $q>2$, yields to a possible extension of our numerical schemes to the case $p<2$. 





\section{Preliminaries and Notation}\label{notation_section}

Let $\Gc=(\edgeset,\nodeset,\omega)$ be a non-oriented graph, where $\edgeset$ is the set of edges, $\nodeset$ is the set of nodes, and $\edgelength$ is a weight defined on the edges of the graph. For each pair of nodes $u$ and $v$ in $\nodeset$ we have that the pair $(u,v)$ is in $\edgeset$ if and only if the pair $(v,u)$ is in $\edgeset$. Also the weights are unique on each edge, i.e., $\edgelength_{uv}=\edgelength_{vu}$. 
We denote by $\incidenceNB\in \R^{|\edgeset|\times |\nodeset|}$ the weighted incidence matrix of the graph, i.e., for each $w\in\nodeset$:
\begin{equation*}
  \incidenceNB\big((u,v),w\big)=
  \edgelength_{uv}\Big(\delta_v(w)-\delta_u(w)\Big)\,,
\end{equation*}
where $\delta_x(\cdot)$ denotes the indicator function of $x$.
Given $\boundary\subset\nodeset$ a generic subset of the vertices,
we say that the pair $(\eigenval,\eigenfunction)$, with $\eigenval\in \R$ and $\eigenfunction$ is a function from $\nodeset$ to $\R$, is a $p$-Laplacian eigenpair with homogeneous Dirichlet boundary conditions in $\boundary$ if it solves the following nonlinear equation:
\begin{equation}\label{eq-p-L-dirichlet}
\begin{cases}
  \frac{1}{2}\big(\incidenceNB^T |\incidenceNB f|^{p-2}\odot
  \incidenceNB f\big)(u)=
  \eigenval |\eigenfunction(u)|^{p-2}\eigenfunction(u)
  \qquad &\forall u\in\internalnodes\\
  f(u)=0 &\forall u\in\boundary\,,
\end{cases}
\end{equation}
where $\odot$ denotes the entrywise or Hadamard vector product. In particular, by analogy with the homogeneous Dirichlet boundary problem on continuous domains and its discretization, we refer to $\boundary$ as the boundary of the graph. However, differently from the continuous setting, where the boundary is inherent to the domain, we point out that the subset of nodes $\boundary$ is chosen arbitrarily.
We adopt the following definition of a connected graph in the presence of a boundary.
\begin{definition}[Connected graph] 
  Given the graph boundary $\boundary\subset\nodeset$, we say that the graph $\Gc$ is connected if the subgraph induced by $\internalnodes$ is connected. 
\end{definition}
If not otherwise stated, in this manuscript we always assume the graph to be connected in the sense of the above definition. Moreover, we denote by $N=|\internalnodes|$ the cardinality of the internal nodes, this also corresponds the dimension of the space of functions on the nodes of $\Gc$ that satisfy the homogeneous Dirichlet boundary conditions.
Indeed, we can identify the functions on the nodes of $\Gc$ that satisfy the Dirichlet boundary conditions in \eqref{eq-p-L-dirichlet} with the functions on the internal nodes $\internalnodes$. We also note that, up to this identification, eq.~\eqref{eq-p-L-dirichlet} is equivalent to the generalized $p$-Laplacian eigenvalue problem defined in~\cite{DEIDDA2023_nod_dom}.
Now, consider $\tilde{\edgeset}$ the subset of the edges obtained by selecting a unique direction for any edge (if $(u,v)\in\tilde{\edgeset}$ then $(v,u)\not\in\tilde{\edgeset}$) and let $\incidence\in \R^{|\tilde{\edgeset}|\times|\internalnodes|}$ be the submatrix of $\incidenceNB$ obtained by sampling the rows corresponding to $\tilde{\edgeset}$ and the columns corresponding to $\internalnodes$. Then for any $f\in\mathcal{H}_0(V)=\{f:\nodeset\rightarrow \R\;|\; f(u)=0\;\forall u\in\boundary\}$, define $\tilde{f}:=f|_{\internalnodes}\in \mathcal{H}(\internalnodes):=\{f:\internalnodes\rightarrow \R\}$ the restriction of $f$ to the internal nodes. Simple computations show that $(\eigenval,\eigenfunction)$ solves \eqref{eq-p-L-dirichlet} if and only if $(\eigenval,\tilde{\eigenfunction})$ solves the following equation:
\begin{equation*}
    \big(\incidence^T |\incidence \tilde{\eigenfunction}|^{p-2}\odot\incidence \tilde{\eigenfunction}\big)(u)=\eigenval |\tilde{f}(u)|^{p-2}\tilde{\eigenfunction}(u) \qquad \forall u\in\internalnodes\,.
\end{equation*}
In particular, a classical argument allows to prove that any $p$-Laplacian eigenpair with homogeneous Dirichlet boundary condition corresponds to a critical point/value of the following $p$-Rayleigh quotient defined on $\mathcal{H}(\internalnodes):=\{f:\internalnodes\rightarrow \R\}$, (see e.g. \cite{Tudisco1}):
\begin{equation*}
\rayl_p(f)=\frac{\|\incidence f\|_p^p}{\|f\|_p^p}=\frac{\sum_{(u,v)\in\tilde{\edgeset}}|\incidence f(u,v)|^p}{\sum_{u\in\internalnodes} |f(u)|^p}\,.
\end{equation*}
Being interested in functions that satisfy Dirichlet boundary conditions, throughout the whole paper, we define the $p$-Laplace operator, or $p$-Laplacian, as follows:
\begin{definition}[$p$-Laplace operator]\label{Def_p-lap}
  \begin{equation*}
    \plap f(u):=\big(\incidence^T |\incidence f|^{p-2}\odot\incidence f\big)(u)\qquad f\in\mathcal{H}(\internalnodes),\;u\in\internalnodes \,.
  \end{equation*}
\end{definition}
We remark that if $\boundary=\emptyset$ our definition of $\plap$ matches the classical definition of the $p$-Laplace operator by means of the incidence matrix~\cite{Tudisco1}. On the other hand, when $\boundary\neq\emptyset$ our $p$-Laplacian is included in the class of the generalized $p$-Laplace operators considered in~\cite{Park2011,DEIDDA2023_nod_dom}.
In addition, we point out that whenever $\boundary\neq\emptyset$, then $\Ker(\incidence)=\{0\}$, since $\Ker(\incidence)\subseteq \Ker(K)=\{\text{Constant functions}\}$.
In the sequel, given $f\in\mathcal{H}(\internalnodes)$ and the corresponding $\tilde{f}\in\mathcal{H}_0(\nodeset)$, for economy of notation and with a small abuse of notation, for any $(u,v)\in\edgeset$, we write
\begin{equation*}
  \incidence f(u,v)=K \tilde{f} (u,v)=\edgelength_{uv}\big(\tilde{f}(v)-\tilde{f}(u)\big)
\end{equation*}
Note that in such a case, by definition of $\tilde{\edgeset}$, since only $\incidence f(v,u)$ is well defined we define $\incidence f(u,v):=-\incidence f(v,u)$ when $(u,v)\not\in \tilde{\edgeset}$.
Then, the $p$-Laplace operator and the corresponding eigenvalue problem can be written as:
\begin{equation}\label{explicit_plap_eig_eq}
  \plap f(u)=\sum_{\substack{v\in \nodeset\\v\sim u}}\edgelength_{uv} |\incidence f(v,u)|^{p-2}\incidence f(v,u)=\eigenval |f(u)|^{p-2}f(u) \qquad \forall u\in\internalnodes,
\end{equation}
where $v\sim u$ means that $(u,v)\in \edgeset$.
We conclude this section by recalling the characterization of the first eigenpair (also called ground state) of the $p$-Laplace operator as the minimum point ($\argmin$) and value ($\min$) of the $p$-Rayleigh quotient $\rayl_p(f)$ \cite{Hua, DEIDDA2023_nod_dom}:
\begin{theorem}[from~\cite{Hua}]\label{Thm:Characterization_of_the_first_eigenvalues}
  Let $(f_1,\eigenval_1):=(\argmin,\min)_{f\in \mathcal{H}(\internalnodes)}\rayl_p(f)$ be the first $p$-Laplacian eigenpair. Then:
  \begin{enumerate}
  \item $\eigenval_1$ is simple, meaning that the associated eigenfunction $f_1$ is unique up to scalar factors;
  \item $f_1$ is the only strictly positive eigenfunction, i.e., if $f$ is an eigenfunction of $\plap$ and $f(v)>0$ for all $v\in\internalnodes$, then $f=f_1$ up to a multiplicative constant.   
  \end{enumerate}  
\end{theorem}
%






\section{An Equivalent Formulation of the $p$-Laplacian Eigenvalue Problem}\label{linear_nonlinear_equivalence_section}
In this section we consider a reformulation of the $p$-Laplacian eigenvalue problem in terms of a constrained weighted Laplacian eigenvalue problem. Using such an equivalence, since the eigenvalues of the corresponding weighted Laplacian are finite, it is possible to assign to every $p$-Laplacian eigenvalue, $\eigenval$, a linear index  defined by the corresponding linear eigenvalue index. We prove that this index, which is theoretically computable, matches the Morse index of $\rayl_p$ in $f$, where $f$ is the $p$-Laplacian eigenfunction corresponding to $\eigenval$. We stress the fact that, here and in the following, we assume $p>2$.

It is possible to observe that the pair $(\eigenval,\eigenfunction)$, solution of the $p$-Laplacian eigenequation \eqref{explicit_plap_eig_eq}, is an eigenpair of the $p$-Laplace operator if and only if $(\eigenval,\eigenfunction)$ solves the following constrained weighted Laplacian Dirichlet eigenvalue problem:
\begin{equation}\label{weighted_lap_eq}
  \begin{cases}
    \lap[\edgeweight]\eigenfunction(u)=\big(\incidence^T \mathrm{diag}({\edgeweight})\incidence\eigenfunction\big)(u)=
    \eigenval\nodeweight(u)\eigenfunction(u) \quad &\forall u\in \internalnodes\\
    {\edgeweight}({uv})=|\incidence\eigenfunction(u,v)|^{p-2} \quad &\forall (u,v)\in\tilde{\edgeset}\\
    {\nodeweight}(u)=|\eigenfunction(u)|^{p-2} &\forall u\in\internalnodes
  \end{cases}\,,
\end{equation}
where $\edgeweight\in\mathcal{M}^+(\edgeset)$ and $\nodeweight\in\mathcal{M}^+(\internalnodes)$, with $\mathcal{M}^+(\edgeset)=\{\edgeweight:\tilde{\edgeset}\rightarrow\R^+, \edgeweight\geq 0\}$ and $\mathcal{M}^+(\internalnodes)=\{\nodeweight:\internalnodes\rightarrow\R^+, \nodeweight \geq 0\}$ denoting the spaces of non-negative measures defined on the edges and on the internal nodes of the graph.
%
Before proceeding with the task of calculating the Morse index of the $p$-Laplacian eigenpairs, we recall some facts about the linear Laplacian generalized eigenvalue problem weighted in $\edgeweight$ and $\nodeweight$. 
Let $\edgeweight\in\mathcal{M}^+(\edgeset)$ and $\nodeweight\in\mathcal{M}^+(\internalnodes)$, we denote by $\diag(\edgeweight)$
and $\diag(\nodeweight)$ the diagonal matrices with entries given by the weights calculated on each edge and each node of the graph, i.e.,
$\diag(\edgeweight)=\diag(\{\edgeweight(uv), (u,v)\in\tilde{\edgeset}\})$ and
$\diag(\nodeweight)=\diag(\{\nodeweight(u), u\in \internalnodes\})$.
Consider the linear generalized eigenvalue problem  
\begin{equation}\label{linear_weighted_eigenpairs}
  \lap[\edgeweight]f(u)=\big(\incidence^T \diag(\edgeweight) \incidence f\big)(u)=\eigenval\diag(\nodeweight) f(u) \quad \forall u\in\internalnodes\,.
\end{equation}
We point out that the $(\edgeweight,\nodeweight)$-weighted Laplacian eigenvalue problem \eqref{linear_weighted_eigenpairs} can be degenerate if $\Ker(\diag(\nodeweight))\cap\Ker(\lap[\edgeweight])$ is non empty.
In this case, there would be only $N-\,\Dim\big(\Ker(\diag(\nodeweight))\cap\Ker(\lap[\edgeweight])\big)$ well defined, possibly equal to infinity, eigenvalues. 
In particular, the well defined generalized eigenvalues can be characterized in terms of the coresponding Rayleigh quotient. To this aim, we introduce the following weighted seminorms on the spaces $\Hc(\tilde{\edgeset}):=\{G:\tilde{\edgeset}\rightarrow \R\}$ and $\Hc(\internalnodes)$: $ \|g\|_{2,\nodeweight}^2=\sum_u \nodeweight_u |g(u)|^2$ where $g\in\Hc(\internalnodes)$, and  $\|G\|_{2,\edgeweight}^2=\sum_{(u,v)\in\tilde{\edgeset}}\edgeweight_{uv}|G(u,v)|^{2}$ where $G\in\Hc(\tilde{\edgeset})$.
The $2$-Rayleigh quotient weighted in $\edgeweight$, $\nodeweight$, given by:
\begin{equation*}
      \rayl_{2,\edgeweight,\nodeweight}(g)=\frac{\|\incidence g\|_{2,\edgeweight}^2}{\|g\|_{2,\nodeweight}
^2}\quad g\in \Hc(\internalnodes)\,,
 \end{equation*}
is well defined on $(\Ker(\diag(\nodeweight))\cap \Ker(\diag(\edgeweight)))^{\perp}$ and takes values in $[0,\infty]$\,.
In particular, the $k$-th well defined eigenvalue can be characterized as the solution of the following saddle-point problem: 
\begin{equation*}
  \eigenval[{(\edgeweight,\nodeweight),k}]
  =\min_{A\in\mathcal{A}_k}\max_{f\in A}\rayl_{2,\edgeweight,\nodeweight}(f)\,,
\end{equation*}
where $\mathcal{A}_k:=\{A\subset \R^{|\internalnodes|}\cap \Ker^{\perp}(\diag(\nodeweight)\cap \lap[\edgeweight])\:|\;\Dim(A)\geq k\}$\,.
%
%

In addition, we will be using the following expanded definition of multiplicity for the well defined $(\edgeweight,\nodeweight)$-Laplacian eigenvalues:
\begin{definition} 
  Let $\eigenval$ be a $(\edgeweight,\nodeweight)$-weighted Laplacian eigenvalue. The multiplicity of $\eigenval$ is
  \begin{equation*}
    \mathrm{mult}(\eigenval)=\Dim\{f\;|\;\lap[\edgeweight]f=\eigenval \diag(\nodeweight) f\}\,.
  \end{equation*}
\end{definition}
Note that, this definition of multiplicity of $\eigenval$ takes into account not only the number of times $\eigenval$ appears in the sequence of the well defined eigenvalues but also the dimension of the subspace $\Ker(\lap[\edgeweight])\cap \Ker(\mathrm{diag}(\nodeweight))$. It finds application in the following result, whose proof is provided in  \cref{sec:appendix} .
\begin{lemma}\label{Lemma_Morse_linear_case}
  Let $(\eigenval[{(\edgeweight,\nodeweight),k}],\eigenfunction[{(\edgeweight,\nodeweight),k}])$ be the $k$-th eigenpair of the generalized $(\edgeweight,\nodeweight)$-{La\-pla\-cian}~\eqref{linear_weighted_eigenpairs} and let $m$ be the multiplicity of $\eigenval[{(\edgeweight,\nodeweight),k}]$. Then:
  \begin{equation*}
    \morse[f](\rayl_{2,\edgeweight,\nodeweight})=k-1\,, \quad \morse[f](-\rayl_{2,\edgeweight,\nodeweight})=N-k-m+1\,,
  \end{equation*}
  where $\morse[f](\rayl_{2,\edgeweight,\nodeweight})$ denotes the Morse index of $\rayl_{2,\edgeweight,\nodeweight}$ evaluated at $\eigenfunction[{(\edgeweight,\nodeweight),k}]$.
\end{lemma}
In essence, the Morse index $\morse[f](\rayl_{2,\edgeweight,\nodeweight})$ is the number of local decreasing directions of $\rayl_{2,\edgeweight,\nodeweight}(\eigenfunction[{(\edgeweight,\nodeweight),k}])$. Formally, the Morse Index can be defined as follows. 
\begin{definition}[Morse Index]
The Morse index of a function $\phi$ at a point $x$, $\morse[x](\phi)$, is defined as the dimension of the largest subspace in which the Hessian matrix of $\phi$ at $x$ is negative definite~\cite{milnor2016morse}.    
\end{definition}
We point out that, sometimes, the Morse index is used only in relation to Morse functions, i.e. functions whose critical points are all non degenerate, but, in general, this is not our case.

We return now to the $p$-Laplacian eigenproblem. Given an eigenpair $(\eigenval,\eigenfunction)$ and the corresponding weights $\edgeweight$ and $\nodeweight$ introduced in \eqref{weighted_lap_eq}, by the definition of $\nodeweight$ we immediately observe that
$\eigenfunction\in\Ker(\diag(\nodeweight))^{\perp}\subset \big(\Ker(\lap[\edgeweight])\cap\Ker(\diag(\nodeweight))\big)^{\perp}$.
Moreover, if we assume without loss of generality that $\|\eigenfunction\|_p=1$, then, by the definition of $\nodeweight$, we have that $\|\eigenfunction\|_{2,\nodeweight}=1$\,.
Thus, if we introduce the spheres 
\begin{equation*}
    S_p:=\{g\in \Hc(\internalnodes)\,|\,\|g\|_p=1\} \quad \text{and}\quad  S_{2,\nodeweight}:=\{g\in \Hc(\internalnodes)\,|\,\|g\|_{2,\nodeweight}=1\}\,,
\end{equation*}
we can state that, if $\eigenfunction\in S_p$, then necessarily $\eigenfunction\in S_{2,\nodeweight}$.
Let $T_{\eigenfunction}(S_p)$ and $T_{\eigenfunction}(S_{2,\nodeweight})$ be the tangent spaces of the two spheres at point $\eigenfunction$. It is not difficult to observe that 
\begin{equation*}
  T_{\eigenfunction}(S_p)=
  \{\xi\,|\;\langle\xi,|\eigenfunction|^{p-2}\odot\eigenfunction\rangle=0\}=
  \{\xi\,|\,\langle\xi,\nodeweight \odot \eigenfunction\rangle=0\}=
  T_{\eigenfunction}(S_{2,\nodeweight})\,.
\end{equation*}
Considering $\rayl_p$ and $\rayl_{2,\edgeweight,\nodeweight}$ as functions defined on the manifolds $S_p$ and $S_{2,\nodeweight}$, the next Lemma shows that it is possible to compare the Morse indices of $\rayl_p$ and $\rayl_{2,\edgeweight,\nodeweight}$ at point $\eigenfunction$.
This allows us to relate $\morse[\eigenfunction](\rayl_p)$ to the linear index of $\eigenval$, i.e., the position of $\eigenval$ in the spectrum of the associated linear eigenvalue problem, $\lap[\edgeweight]\eigenfunction=\eigenval \diag(\nodeweight) f$.
\begin{proposition}\label{increasing_directions}
  Given an eigenpair $(\eigenval,\eigenfunction)$ of the $p$-Laplacian  and the  weights $\nodeweight=|\eigenfunction|^{p-2}$ and $\edgeweight=|\incidence\eigenfunction|^{p-2}$, assume that $(\eigenval,\eigenfunction)=\big(\eigenval[{(\edgeweight, \nodeweight),k}],\eigenfunction[{(\edgeweight,\nodeweight),k}]\big)$ have multiplicity $m$. Then:
  \begin{align*}
    \morse[f](\rayl_p)
    &=\morse[f](\rayl_{2,\edgeweight,\nodeweight})=k-1\,,\\
    \morse[f](-\rayl_p)
    &=\morse[f](-\rayl_{2,\edgeweight,\nodeweight})=N-k-m+1\,.
  \end{align*}
\end{proposition}
\begin{proof}

Note that both $\rayl_p$ and $\rayl_{2,\edgeweight,\nodeweight}$ are zero homogeneous, thus we can limit to study their Morse index on the manifold $S_p$ and $S_{2,\nodeweight}$, respectively. 
In particular, since $T_{\eigenfunction}(S_p)=T_{\eigenfunction}(S_{2,\nodeweight})$, to prove the Proposition it is enough to show that $\forall \xi\in T_{\eigenfunction}(S_p)=T_{\eigenfunction}(S_{2,\nodeweight})$ we have:
\begin{equation}\label{eq_Morse_proposition}
  \frac{\partial^{2}}{\partial \epsilon^{2}}
  \bigg(
  \frac{\|\incidence(\eigenfunction+\epsilon \xi)\|_p^p}{\|\eigenfunction+\epsilon \xi\|_p^p}
  \bigg)\bigg|_{\epsilon=0}
  =\frac{p(p-1)}{2} \frac{\partial^{2}}{\partial \epsilon^{2}}
  \bigg(
  \frac{\|\incidence(\eigenfunction+\epsilon \xi)\|_{2,\edgeweight}^2}{\|\eigenfunction+\epsilon \xi\|_{2,\nodeweight}^2}
  \bigg)\bigg|_{\epsilon=0}\,,
\end{equation}
i.e. the Hessian of the two functions behaves analogously on the tangent space of the two manifolds. The equality in \eqref{eq_Morse_proposition} is proved exploiting the expression of the second derivatives and recalling that the first derivative of both the Rayleigh quotients is zero in $f$, since this is an eigenvector. Finally, the conclusion follows from \Cref{Lemma_Morse_linear_case}. We refer to \cref{sec:appendix} for all the details.
\end{proof}
We would like to observe that the results proved in this section show that, given a $p$-Laplacian eigenpair, the linear index of the $(\mu,\nu)$-eigenvalue provides information about the behaviour of the $p$-Rayleigh quotient in a neighborhood of the eigenfunction, i.e. it tells us what kind of critical point/value a $p$-Laplacian eigenpair is. We use this information in the next sections to characterize a $p$-Laplacian eigenpair computed as a solution of the linearized eigenvalue problem \eqref{weighted_lap_eq}, where the index $k$ of the linear eigenvalue $\lambda_{(\edgeweight,\nodeweight),k}$, is fixed.
%


\section[Energy functions]{Nonlinear Eigenpairs as Critical Points of a Family of Energy Functions} 
The results of the previous section suggest to use the $(\mu,\nu)$-eigenvalue problems as much as possible. Following this suggestion and taking inspiration from the energy function introduced in~\cite{facca2}, we propose a family of energy functions $\Eps_{p,k}$, defined on $\mathcal{M}^+(\edgeset)\times\mathcal{M}^+(\internalnodes)$ and indexed by $k$, such that their critical points identify $p$-Laplace eigenpairs. The $k$-th member of this family is given by
%
\begin{equation}\label{Higher_energy_functions}
  \Eps_{p,k}(\edgeweight,\nodeweight):=
  \frac{1}{\eigenval[{(\edgeweight,\nodeweight),k}]}
  +\mass_{\edgeset,p}(\edgeweight)-\mass_{\nodeset,p}(\nodeweight)\,,
\end{equation}
where $\eigenval[{(\edgeweight,\nodeweight),k}]$ is the k-th well defined eigenvalue of the weighted Laplacian eigenvalue problem~\eqref{linear_weighted_eigenpairs} and the ``mass functions'' $\mass_{\nodeset,p}(\nodeweight)$, $\mass_{\edgeset,p}(\edgeweight)$ are given by
\begin{equation*}
    \mass_{\nodeset,p}(\nodeweight):=\frac{p-2}{p}\sum_{u\in\internalnodes}\nodeweight_u^{\frac{p}{p-2}}\quad \text{and} \quad  \mass_{\edgeset,p}(\edgeweight):=\frac{p-2}{p}\sum_{(u,v)\in\tilde{\edgeset}}\edgeweight_{uv}^{\frac{p}{p-2}}\,.
\end{equation*}%
We first state the main results of this section and discuss their significance, postponing the proofs to the end of the section. The first theorem shows that any differentiable saddle point of such energy functions corresponds to a $p$-Laplacian eigenpair.
\begin{theorem}\label{thm:Smooth_saddle_points_correspond_to_p_eigenpairs}
  Let $(\optedgeweight,\optnodeweight)\in \mathcal{M}^+(\edgeset)\times \mathcal{M}^+(\internalnodes)$ be a differentiable saddle point of the function $\Eps_{p,k}(\edgeweight,\nodeweight)$. Then, $(\eigenval[{(\optedgeweight,\optnodeweight),k}]^{p/2},\eigenfunction[{(\optedgeweight,\optnodeweight),k}])$ is a $p$-Laplacian eigenpair such that $\morse_{\eigenfunction[{(\optedgeweight,\optnodeweight),k}]}(\rayl_p)=k-1$, and $\morse_{\eigenfunction[{(\optedgeweight,\optnodeweight)}]}(-\rayl_p)=N-k$.
\end{theorem}
The second theorem asserts that if the boundary of the graph is not empty, $\boundary\neq \emptyset$, for $k=1$ the hypothesis of differentiability can be removed. Indeed, $\Eps_{p,1}$ has always a unique saddle point which corresponds to the unique first eigenpair of the $p$-Laplacian.
\begin{theorem}\label{Thm_saddle_point_1st_pp_eigenpair}
  Let $\boundary\neq\emptyset$. Then the function $\Eps_{p,1}(\edgeweight,\nodeweight)$ admits a unique saddle point   
  \begin{equation*}
    (\edgeweight^*,\nodeweight^*)=
    \argmax_{\nodeweight\in\mathcal{M}^+(\internalnodes)\setminus\{0\}}
    \argmin_{\edgeweight\in\mathcal{M}^+(\edgeset)}
    \Eps_{p,1}(\edgeweight,\nodeweight)\,.
  \end{equation*}
  Moreover, if $\eigenval[{(\optedgeweight,\optnodeweight),1}]$ is the first eigenvalue of the Laplacian eigenvalue problem~\eqref{linear_weighted_eigenpairs} weighted in $(\optedgeweight,\optnodeweight)$, then there exists an associated eigenfunction $\eigenfunction[{(\optedgeweight,\optnodeweight),1}]$ such that  $(\eigenval[{(\optedgeweight,\optnodeweight),1}]^{p/2},\eigenfunction[{(\optedgeweight,\optnodeweight),1}])$ equals the first $p$-Laplacian eigenpair.
\end{theorem}  
Observe that, in general, the $k$-th energy function $\Eps_{p,k}$ may not be well defined on the boundary of $\mathcal{M}^+(\edgeset)\times\mathcal{M}^+(\internalnodes)$ since in this case well defined $k$-th eigenvalues may not exist. However, $\Eps_{p,1}$ encounters this problem only in the degenerate case $(\nodeweight,\edgeweight)=(0,0)$.
We would like to remark that the assumption $\boundary\neq \emptyset$ is not restrictive, indeed in the case $\boundary=\emptyset$ $\Ker(\plap)=\Span\{\underline{1}\}$, where $\underline{1}$ is the constant function equal to $1$ on the nodes of the graph.
The differentiability hypothesis in the above theorems is nontrivial since lack of continuity of the energy functions in~\eqref{Higher_energy_functions} may occur when both $\edgeweight\in\partial\mathcal{M}^+(\edgeset)$ and $\nodeweight\in\partial\mathcal{M}^+(\internalnodes)$, where $\partial \mathcal{M}^+(\edgeset)$ and $\partial \mathcal{M}^+(\internalnodes)$ denote the boundary of $\mathcal{M}^+(\edgeset)$ and $\mathcal{M}^+(\internalnodes)$. Indeed, in this case, the generalized Laplacian eigenvalues may no longer be continuous~\cite{templates}.
Moreover, the function $\Eps_{p,k}(\edgeweight,\nodeweight)$ is not differentiable whenever $\eigenval[{(\edgeweight,\nodeweight),k}]$ is not simple~\cite{kato2013perturbation}.
In Fig.~\ref{Fig-zeroonedge} we provide an example of this degeneracy in a $p$-Laplacian eigenpair problem.

\begin{figure}
  \centering
  \begin{tikzpicture}[inner sep=1.5mm, scale=.5, thick]

    \node (1) at (0,2) [circle,draw] {\textbf{B}};
    \node (2) at (3,0) [circle,draw] {1};
    \node (4) at (3,4) [circle,draw] {2};
    \node (3) at (6,2) [circle,draw] {3};

    \node (5) at (9,2) [circle,draw] {4};
    \node (6) at (12,0)  [circle,draw] {5};
    \node (7) at (12,4) [circle,draw] {6};
    \node (8) at (15,2) [circle,draw] {\textbf{B}};

    \draw [-] (1.south) -- (2.west);
    \draw [-] (1.north) -- (4.west);
    \draw [-] (4.east) -- (3.north);
    \draw [-] (4.south) -- (2.north);
    \draw [-] (2.east) -- (3.south);
    \draw [-] (3.east) -- (5.west);
    \draw [-] (5.south) -- (6.west);
    \draw [-] (5.north) -- (7.west);
    \draw [-] (7.east) -- (8.north);
    \draw [-] (8.south) -- (6.east);
    \draw [-] (6.north) -- (7.south);

  \end{tikzpicture}
  \caption{A graph with non-simple first eigenvalue. Assume $\nodeweight_u=1\; \forall u\in\internalnodes$, then the graph is symmetric and the first eigenfunction of $\plap$, $\eigenfunction[{[p,p],1}]$, is unique and necessarily agrees with the symmetry of the graph. This means that $\incidence\eigenfunction[{[p,p],1}](3,4)=0$ and thus the density $\edgeweight=|\incidence\eigenfunction[{[p,p],1}]|^{p-2}$ of eq.~\eqref{linear_weighted_eigenpairs} is zero on the edge $(3,4)$, splitting $\Gc$ in two connected components. As a result, $\eigenval[{(\edgeweight,\nodeweight),1}]$ is not simple and $\Eps_{p,1}$ is not differentiable.} \label{Fig-zeroonedge}
\end{figure}

The last preliminary result needed to tackle the proof of Theorem~\ref{thm:Smooth_saddle_points_correspond_to_p_eigenpairs}, is the following technical Lemma, which, assuming $\eigenval[{(\edgeweight,\nodeweight),k}]$ differentiable at $(\edgeweight^*, \nodeweight^*)$, provides a classical characterization of the derivatives of $\eigenval[{(\edgeweight,\nodeweight),k}]$ with respect to  $\edgeweight$ and $\nodeweight$.
\begin{lemma}\label{lemma_eigenvalues_derivative}
  Let $\eigenval^*_k=\eigenval[{(\edgeweight,\nodeweight),k}]$ be
  differentiable in $(\edgeweight^*,\nodeweight^*)$, then that the corresponding eigenfunction $\eigenfunction^*_k=\eigenfunction[{(\edgeweight,\nodeweight),k}]$ is unique and:
  \begin{equation*}
    \partial_{\edgeweight}
    \Big(
    (\eigenval^*_k)^{-1}
    \Big)
    = -\frac{|\incidence\eigenfunction^*_k|^2}
    {(\eigenval^*_k)^2\|\eigenfunction^*_k\|_{2,\nodeweight^*}^2} 
    \qquad \mathrm{and} \qquad
    \partial_{\nodeweight}
    \Big(
    (\eigenval^*_k)^{-1}
    \Big)
    = \frac{|\eigenfunction^*_k|^2}
    {\|\incidence\eigenfunction^*_k\|_{2,\edgeweight^*}^2}\,.
  \end{equation*}
\end{lemma}
\begin{proof}
  The proof uses the fact that if an eigenvalue is differentiable, then it is necessarily simple~\cite{kato2013perturbation}, we refer to \cref{sec:appendix} for all the computations.
\end{proof}
Next we present the proof of \Cref{thm:Smooth_saddle_points_correspond_to_p_eigenpairs}
\begin{proof}[Proof of Theorem~\ref{thm:Smooth_saddle_points_correspond_to_p_eigenpairs}]
  The generalized $k$-th $(\mu,\nu)$-Laplacian eigenpair is a function of $\mu$ and $\nu$. To simplify notation, when no ambiguity arises, in this proof we write $\eigenval_k$ and $\eigenfunction_k$ with no explicit reference to the dependence upon $(\mu,\nu)$. In addition, we write $\eigenval^*_k:=\eigenval[{(\optedgeweight,\optnodeweight),k}]$ and $\eigenfunction^*_k:=\eigenfunction[{(\optedgeweight,\optnodeweight),k}]$, i.e., $\eigenval^*_k$ and $\eigenfunction^*_k$ are the $k$-th $(\edgeweight,\nodeweight)$-Laplacian eigenvalue and eigenfunction evaluated at optimality.

  Thanks to Lemma~\ref{lemma_eigenvalues_derivative}, the KKT conditions for the saddle points of the energy function $\Eps_{p,k}(\edgeweight,\nodeweight)$ can be written as:
  \begin{equation}\label{K-K-T_Eps_k}
    \begin{cases}
      \lap[\optedgeweight]\eigenfunction^*_k
      =\eigenval^*_k \diag(\optnodeweight) \eigenfunction^*_k\\[.6em]
      -\big(|\incidence\eigenfunction^*_{k}(u,v)|^2\big/(\eigenval^*_k)^2\|\eigenfunction^*_k\|_{2,\optnodeweight}^2\big)
      +{\optedgeweight_{uv}}^{\frac{2}{p-2}}-c_{uv}=0
      \quad &\forall (u,v)\in\tilde{\edgeset}\\[.6em]
      \big(|\eigenfunction^*_k(v)|^2\big/
      \|\incidence\eigenfunction^*_k\|_{2,\optedgeweight}^2\big)
        -{\optnodeweight_v}^{\frac{2}{p-2}}+s_v=0
      &\forall v\in\internalnodes\\
      c_{uv}\optedgeweight_{uv}=0\,, \quad c_{uv}\geq 0
      &\forall (u,v)\in\tilde{\edgeset} \\ 
      s_v\optnodeweight_u=0\,,\quad s_v\geq 0
      &\forall v\in\internalnodes      
    \end{cases}
  \end{equation}
  where $\tilde{\edgeset}$ is the subset of the edges obtained by selecting a unique direction for any edge (see Section~\ref{notation_section}). The constants $\{c_{uv}\}_{(u,v)\in\tilde{\edgeset}}$ and $\{s_v\}_{v\in\internalnodes}$ are suitable families of Lagrange multipliers. 
  Since $c_{uv}\geq 0$ whenever $\optedgeweight_{uv}=0$, the second equation of the system \ref{K-K-T_Eps_k} admits only the solution $\incidence\eigenfunction^*_k(u,v)=0=c_{uv}$, when $\optedgeweight_{uv}=0$.
  Analogously, $\optnodeweight_v=0$ implies $\eigenfunction^*_k(v)=s_v=0$.
  Hence equation~\eqref{K-K-T_Eps_k} yields the following equalities:
  \begin{equation}\label{eq_p_optimal_density}
      \displaystyle{
      \optedgeweight=
      \frac{|\incidence\eigenfunction^*_k|^{p-2}}
        {(\eigenval^*_k)^{p-2}\|\eigenfunction^*_k\|_{2,\optnodeweight}^{p-2}}
      }\,,\qquad
      \displaystyle{
      \optnodeweight=
      \frac{|\eigenfunction^*_k|^{p-2}}
      {\|\incidence\eigenfunction^*_k\|_{2,\optedgeweight}^{p-2}}}\,, \quad\text{where }\quad 
      \displaystyle{
      \lap_{\optedgeweight}\eigenfunction^*_k
      =\eigenval^*_k\optnodeweight\eigenfunction^*_k
      }\,.
  \end{equation}
  Now we can write:
  \begin{equation*}\label{cmu_e_cnu}
    \begin{cases}
      \displaystyle{
      \optedgeweight=
      c_{\edgeweight}|\incidence\eigenfunction^*_k|^{p-2}
      }\\[.5em]
      \displaystyle{
      \optnodeweight=c_{\nodeweight}|\eigenfunction^*_k|^{p-2}
      }
    \end{cases} 
    \quad 
    \text{with}
    \quad 
    \begin{cases}
      \displaystyle{ c_{\edgeweight}
      =(\eigenval^*_k)^{2-p}\|\eigenfunction^*_k\|_{2,\optnodeweight}^{2-p}
      }\\[.5em]
      \displaystyle{
      c_{\nodeweight}=\|\incidence\eigenfunction^*_k\|_{2,\optedgeweight}^{2-p}
      }
    \end{cases}\,.
  \end{equation*}
  Dividing the second equation in the previous expression by the first one we obtain:
  \begin{equation*}
    \frac{c_{\nodeweight}}{c_{\edgeweight}}
    ={\eigenval_k^*}^{p-2}
    \left(
      \frac{\|\eigenfunction^*_k\|_{2,\optnodeweight}^2}
      {\|\incidence\eigenfunction^*_k\|_{2,\optedgeweight}^2}
    \right)^{\frac{p-2}{2}}
    =(\eigenval^*_k)^{\frac{p-2}{2}}\,.
  \end{equation*}
  Finally, replacing the 
  expressions for $\optedgeweight$ and $\optnodeweight$ from \eqref{cmu_e_cnu} in the last equation of~\eqref{eq_p_optimal_density}, dividing by $c_{\edgeweight}$, and using the ratio ${c_{\nodeweight}}/{c_{\edgeweight}}$ just calculated, we obtain:
  \begin{equation*}
    \sum_{ v\sim u}\edgelength_{uv}
    |\incidence\eigenfunction^*_k(v,u)|^{p-2}
    \incidence\eigenfunction^*_k(v,u)
    =
    (\eigenval^*_k)^{\frac{p}{2}}
    |\eigenfunction^*_k(u)|^{p-2}\eigenfunction^*_k(u)\,,
  \end{equation*}
  which shows that $(\eigenval[{(\optedgeweight,\optnodeweight),k}]^{p/2},\eigenfunction[{(\optedgeweight,\optnodeweight),k}])$ is a $p$-Laplacian eigenpair. 
  To conclude, we observe that the hypothesis that $(\optedgeweight,\optnodeweight)\in\mathcal{M}^+(\edgeset)\times\mathcal{M}^+(\internalnodes)$ is a differentiable saddle point of the function $\Eps_{p,k}(\edgeweight,\nodeweight)$ is equivalent to assuming that $\eigenval[{(\edgeweight,\nodeweight),k}]$ is a simple eigenvalue of the generalized Laplacian eigenvalue problem~\eqref{linear_weighted_eigenpairs}. Indeed, an eigenvalue is differentiable if and only if it is simple \cite{kato2013perturbation}. Then note that $(\lambda, f)$ is an eigenpair satisfying \eqref{linear_weighted_eigenpairs} with  $(\edgeweight,\nodeweight)=(\edgeweight^*,\nodeweight^*)$ if and only if $((c_{\nu}/c_{\mu})\lambda, f)$ satisfies \eqref{linear_weighted_eigenpairs} with $(\edgeweight,\nodeweight)=(|\grad f_k^*|^{p-2},|f_k^*|^{p-2})$. Thus, $((\lambda_k^*)^{p/2}, f_k^*)$ is the $k$-th eigenpair of the generalized Laplacian eigenvalue problem corresponding to $(\edgeweight,\nodeweight)=(|\grad f_k^*|^{p-2},|f_k^*|^{p-2})$, with the eigenvalue being also simple. Then, \Cref{increasing_directions} shows that $\eigenfunction[{(\optedgeweight,\optnodeweight),k}]$ is a $p$-Laplacian eigenfunction such that $\morse_{\eigenfunction[{(\optedgeweight,\optnodeweight),k}]}(\rayl_p)=k-1$ and $\morse_{\eigenfunction[{(\optedgeweight,\optnodeweight)}]}(-\rayl_p)=N-k$.
\end{proof}
Next, we turn our attention to the proof of Theorem \ref{Thm_saddle_point_1st_pp_eigenpair} and the necessary preliminary results. The proof of the theorem is subdivided in two parts. The first part works on the weighted $[p,2]$-Laplacian eigenvalue problem and the second part extends these results to the $p$-Laplacian (or $[p,p]$-Laplacian) eigenproblem. Here we use square brackets to avoid confusion with the weighted $(\mu,\nu)$ generalized Laplacian eigenproblem used before.
Because the $[p,2]$-Laplacian is of independent interest~\cite{franzina2010existence, otani1984certain, Bungert4} we decided to subdivide these two parts into two subsections.
From know on, if not otherwise stated, we assume $\boundary\neq\emptyset$.



\subsection{The $[p,2]$-Laplacian Eigenvalue Problem}\label{p,2_Laplacian Section}

Let $\nodeweight \in \mathcal{M}^{+}(\internalnodes)$ be a density on the nodes with $\nodeweight\neq 0$ and consider the following $[p,2]$-Rayleigh quotient, which possibly can take the value $+\infty$:
\begin{equation*}
  \rayl_{p,2,\nodeweight}(\eigenfunction)
  =\frac{\|\nabla\eigenfunction\|_p^p}
      {\|\eigenfunction\|_{2,\nodeweight}^{p}}
  =
  {
    \displaystyle{
      \sum_{(u,v)\in \tilde{\edgeset}}|\incidence\eigenfunction(uv)|^p
    }
  }
  \big/
  {
    \big(
    \displaystyle{
      \sum_{u\in\internalnodes}\nodeweight_u|\eigenfunction(u)|^2
    }
    \big)^{\frac{p}{2}}
  }\,.
\end{equation*}
We assume $\rayl_{p,2,\nodeweight}$ to be defined on the domain $\mathcal{H}(\internalnodes)$ and we name its critical point equation the $[p,2]$-Laplacian eigenvalue equation weighted in $\nodeweight$:
\begin{equation}\label{eq:p2-L-dirichlet}
  (\plap \eigenfunction)(u)
  =\eigenval\,\nodeweight_u\,
  \|\eigenfunction\|_{2,\nodeweight}^{p-2} \eigenfunction(u)
  \quad \forall u\in \internalnodes \, .
\end{equation}
We provide now a characterization of the first eigenpair of the $[p,2]$-Laplacian as the minimal value and the minimum point of $\rayl_{p,2,\nodeweight}$.
In particular, we use the notation $(\eigenval_{[p,2,\nodeweight],1},\eigenfunction_{[p,2,\nodeweight],1})$ to indicate the $1$-st weighted $[p,2]$-eigenpair. In addition, we denote by $(\eigenval_{[p,p],1},\eigenfunction_{[p,p],1})$ the first eigenpair of the $p$-Laplacian discussed in the previous sections (see Theorem \ref{Thm:Characterization_of_the_first_eigenvalues}).
We would like to note that the characterization of $(\eigenval_{[p,2,\nodeweight],1},\eigenfunction_{[p,2,\nodeweight],1})$ we are going to prove is a simple extension of the characterization of the classical $p$-Laplacian eigenpair proposed in~\cite{Hua} and already used in Thm~\ref{Thm:Characterization_of_the_first_eigenvalues}. Moreover the continuous analogue result of our result is well known to hold \cite{idogawa1995first}. Thus the proof is moved to \cref{sec:appendix}.
%
In particular, the Theorem states that the first eigenvalue of the $[p,2]$-Laplacian is simple and positive and the corresponding unique first eigenfunction is the only one that is strictly positive on all internal nodes.
\begin{theorem}\label{p2_first_eigen_charact}
  Let $\nodeweight\neq 0$ and $\Gc$ be a connected graph. If $(\eigenval_{[p,2,\nodeweight],1},\eigenfunction_{[p,2,\nodeweight],1})$ is  a first eigenpair of the $[p,2]$-Laplacian, then $\eigenval_{[p,2,\nodeweight],1}\geq0$ and $\eigenfunction_{[p,2,\nodeweight],1}(u)>0$  for all  $u\in \internalnodes$. Moreover $\eigenval_{[p,2,\nodeweight],1}$ is simple and $\eigenfunction_{[p,2,\nodeweight],1}$ is the unique eigenfunction strictly greater than zero on every internal node.
\end{theorem}
%

\begin{remark} 
   
  Observe that the same argument provides a characterization also for the first eigenpair of the $(\edgeweight,\nodeweight)$-Laplacian eigenvalue problem. In particular, let $\edgeweight\in \mathcal{M}^+(\edgeset)$ and let $\Gc_{\edgeweight}$ be the subgraph of $\Gc$ obtained by removing the edges where $\edgeweight=0$. Assume that $\Gc_{\edgeweight}$ is connected and observe that the first well defined eigenvalue can be written as:
  \begin{equation*}
    \eigenval_{(\edgeweight,\nodeweight),1}=
    \min_{\|\eigenfunction\|_{\nodeweight}=1}
    \rayl_{2,\edgeweight,\nodeweight}(\eigenfunction)\,.
  \end{equation*}
  The same proof of Theorem \ref{p2_first_eigen_charact} shows that $\eigenval_1(\edgeweight,\nodeweight)$ is simple and the corresponding eigenfunction $\eigenfunction_1$ is uniquely characterized by the property of being strictly positive on any node of $\Gc_{\edgeweight}$. Finally note that if $\Gc_{\edgeweight}$ is not connected, even  if the ``if and only if'' condition does not hold, it is still possible to show that if we find a function $\eigenfunction$ that satisfies the $(\edgeweight,\nodeweight)$-Laplacian eigenvalue equation and that is strictly positive on the internal nodes, then necessarily the corresponding eigenvalue is the first one, as the following corollary states.
\end{remark}

\begin{corollary}\label{Corollary_1st_eigenpair_laplacian}
  Given $\edgeweight\in\mathcal{M}^+(\edgeset)$ and $\nodeweight\in\mathcal{M}^+(\internalnodes)$ with $\edgeweight, \nodeweight \neq 0$. If $(\eigenval,\eigenfunction)$ is an eigenpair of the $(\edgeweight,\nodeweight)$-Laplacian such that $\eigenfunction(u)>0$ for any $v\in \internalnodes$, then $\eigenval=\eigenval[{(\edgeweight,\nodeweight),1}]$\,.
\end{corollary}
\begin{proof}
  The proof easily follows by observing that, even if the induced graph has been disconnected, $\Gc_{\edgeweight}=\cup\Gc_i$ with $\Gc_i$ disjoint, the $(\edgeweight,\nodeweight)$-spectrum is given by the union of the $(\edgeweight|_{\Gc_i},\nodeweight|_{\Gc_i})$-spectra.
  Moreover, for any $\Gc_i$ where the $(\edgeweight|_{\Gc_i},\nodeweight|_{\Gc_i})$-Laplacian eigenvalue problem is defined, i.e. $(\edgeweight|_{\Gc_i},\nodeweight|_{\Gc_i})\neq (0,0)$, the first eigenfunction is characterized by the property of being strictly positive on $\Gc_i$, i.e. $\eigenfunction[{(\edgeweight|_{\Gc_i},\nodeweight|_{\Gc_i}),1}](u)>0$ for all $u\in\Gc_i$.
Thus, if $\eigenfunction$ is an eigenfunction on $\Gc$ and $\eigenfunction>0$, necessarily $\eigenfunction=\sum_{i}\alpha_i\eigenfunction[{(\edgeweight|_{\Gc_i},\nodeweight|_{\Gc_i}),1}]$ for some $\{\alpha_i\}_j>0$, i.e., $\eigenfunction$ corresponds to the first eigenvalue on any connected component.
\end{proof}




\subsubsection{The $[p,2]$-Laplacian eigenproblem as a $(\edgeweight,\nodeweight)$-Laplacian eigenproblem}

Analogously to the $p$-Laplacian eigenvalue problem discussed in Section \ref{linear_nonlinear_equivalence_section}, also the $[p,2]$-Laplacian eigenvalue problem can be reformulated in terms of a constrained weighted Laplacian eigenvalue problem.
To this aim, we first rewrite the eigenvalue equation~\eqref{eq:p2-L-dirichlet} as:
\begin{equation*}
    \incidence^T
    \left(|\incidence\eigenfunction|^{p-2}\odot
      \incidence \eigenfunction
    \right)(u)
  =\eigenval\,\nodeweight_u \|\eigenfunction\|_{2,\nodeweight}^{p-2}\eigenfunction(u)\qquad \forall u\in\internalnodes. 
\end{equation*}
Dividing both terms by $\|\eigenfunction\|_{2,\nodeweight}^{p-2}$, it is possible to observe that $(\eigenval,\eigenfunction)$ is an eigenpair of the $[p,2]$-Laplacian if and only if $(\eigenval,\eigenfunction)$ is an eigenpair of the constrained weighted Laplacian problem, i.e., it is a solution of the following equation:
\begin{equation*}
  \begin{cases}
    \lap_{\mu}\eigenfunction(u)
    :=
    \incidence^T
    \left(\diag(\edgeweight)\incidence \eigenfunction\right)
    (u)
    = \eigenval \nodeweight_u \eigenfunction(u)
    \quad &\forall u\in \internalnodes\\[0.3em]
    \displaystyle{
      \edgeweight_{uv}
      =|\incidence\eigenfunction(u,v)|^{p-2}/\|\eigenfunction\|_{2,\nodeweight}^{p-2}
      \geq 0
    }
    \quad &\forall \,(u,v)\in\tilde{\edgeset}
  \end{cases}
\end{equation*}




\subsection{Energy Function for the first eigenpair of the $[p,2]$-Laplacian} 

In this section we introduce a convex energy function whose minimum can be proved to correspond to the unique first eigenpair of the $[p,2]$-eigenvalue problem weighted in $\nodeweight$. The results and the techniques presented here are the starting point to prove Theorem~\ref{Thm_1st_pp_characterization_as_crtcl_point}.

Given a fixed density $\nodeweight\in\mathcal{M}^+(\internalnodes)$ with $\nodeweight\neq 0$, consider the following energy function:
\begin{equation*}
  \begin{aligned}
    \lyap[1,\edgeset](\edgeweight)
    =&
    \dfrac{1}{\eigenval[{(\edgeweight,\nodeweight),1}]}
    +{\mass}_{\edgeset,p}(\edgeweight)
    =
    \sup_{\|\eigenfunction\|_{2,\nodeweight}=1}
    \dfrac{\|\eigenfunction\|_{2,\nodeweight}^2}
    {\|\incidence\eigenfunction\|^2_{2,\edgeweight}}
    +\frac{p-2}{p}
    \sum_{(u,v)\in\tilde{\edgeset}}\edgeweight[uv]^{\frac{p}{p-2}}
    \end{aligned}
\end{equation*}
Observe that $\lyap[1,\edgeset]$ is the only part of $\Eps_{p,1}$ in eq.~\eqref{Higher_energy_functions} that depends on $\edgeweight$.
In the following Theorem we prove that the energy function $\lyap[1,\edgeset](\edgeweight)$ admits a unique minimizer, $\optedgeweight$, and that the first eigenfunction of $\lap[\optedgeweight]$ corresponds to the unique first eigenpair of the $[p,2]$-Laplacian. 

\begin{theorem}\label{Thm_p-2_Lyap_Minimizer}
  Let $\nodeweight\in\mathcal{M}^+(\internalnodes)$ with $\nodeweight\neq 0$ and assume $\optedgeweight$ is a minimum point of $\lyap[1,\edgeset](\edgeweight)$ on $\mathcal{M}^+(\edgeset)$. Given $\eigenval^*_1=\eigenval_{(\optedgeweight,\nodeweight),1}$, there exist a $(\optedgeweight,\nodeweight)$-eigenfunction $\eigenfunction^*_1$ associated to $\eigenval^*_1$ such that
  %
  %
  $\big((\eigenval^*_1)^{p-1}, \eigenfunction^*_1\big)\,,$ is the first $[p,2]$-eigenpair, i.e.:
  \begin{equation*}
      (\eigenval^*_1)^{p-1}
      = \eigenval[{[p,2,\nodeweight],1}]
      \qquad\mbox{and}\qquad
      \eigenfunction^*_1=
      \eigenfunction[{[p,2,\nodeweight],1}]\,.
  \end{equation*}
  Moreover   $\lyap[1,\edgeset](\optedgeweight)
    = ((2p-2)/p)\eigenval[{[p,2,\nodeweight],1}]^{-1/(p-1)}  $.
\end{theorem}
\begin{proof}
We first note that as a composition of convex functions, $\lyap[1,\edgeset]$  is strictly convex in $\mathcal{M}^+(\edgeset)$ and thus it admits a unique minimizer $\optedgeweight$. Moreover, using the characterization of $\eigenval[{(\edgeweight,\nodeweight),1}]$ by means of the $(\edgeweight,\nodeweight)$-Rayleigh quotient $\rayl_{2,\edgeweight,\nodeweight}$, the minimum problem of the function $\lyap_{1,\edgeset}$ can be written as a saddle point problem, i.e.:
  \begin{equation*}
    \min_{\edgeweight\in \mathcal{M}^+(\edgeset)}\lyap_{1,\edgeset}(\edgeweight)
    =\min_{\edgeweight\in\mathcal{M}^+(\edgeset)}
    \max_{\|\eigenfunction\|_{2,\nodeweight}=1}
    \frac{\|\eigenfunction\|_{2,\nodeweight}^2}
    {\|\incidence\eigenfunction\|^2_{2,\edgeweight}}
    +\frac{p-2}{p}
    \sum_{(u,v)\in\tilde{\edgeset}}\edgeweight[uv]^{\frac{p}{p-2}}\,.
  \end{equation*}
  From the minmax inequality, we can write:
  \begin{equation}\label{lower_inf_sup_inequality}
    \min_{\edgeweight\in\mathcal{M}^+(\edgeset)}
    \lyap_{1,\edgeset}(\edgeweight)\geq
    \max_{\|f\|_{2,\nodeweight}=1}\min_{\edgeweight\in\mathcal{M}^+(\edgeset)}
    \frac{\|\eigenfunction\|_{2,\nodeweight}^2}
    {\|\incidence\eigenfunction\|^2_{2,\edgeweight}}
    +
    \frac{p-2}{p}
    \sum_{(u,v)\in\tilde{\edgeset}}\edgeweight[uv]^{\frac{p}{p-2}}\,.
  \end{equation}
  Now, for a fixed $\eigenfunction$ with $\|\eigenfunction\|_{2,\nodeweight}=1$, it is possible to compute the weight $\edgeweight^{\eigenfunction}$ that realizes the minimum, i.e.:
  \begin{equation*}
    \edgeweight^{\eigenfunction}=\argmin_{\edgeweight\in \mathcal{M}^+(\edgeset)}
    \frac{1}{\|\incidence\eigenfunction\|^2_{2,\edgeweight}}
    +\frac{p-2}{p}
    \sum_{(u,v)\in\tilde{\edgeset}}\edgeweight[uv]^{\frac{p}{p-2}}\,.
  \end{equation*}
  Indeed, the KKT conditions for this constrained minimization problem are:
  \begin{equation}\label{Thm_1st_energy_p2_KKT}
    \begin{cases}
      -\dfrac{|\incidence f(u,v)|^2}
      {\|\incidence\eigenfunction\|_{\edgeweight^{\eigenfunction}}^4}
      +\left(\edgeweight^{\eigenfunction}_{uv}\right)^{\frac{2}{p-2}}
      -
      c_{uv}=0 \qquad
      &\forall (u,v)\in\tilde{\edgeset} \\
      c_{uv}\geq 0
      \quad \text{and}\quad
      c_{uv}\edgeweight^{\eigenfunction}_{uv}=0
      &\forall (u,v)\in\tilde{\edgeset}  
    \end{cases}
  \end{equation}
  where $\{c_{uv}\}$ is a family of edge-wise Lagrange multipliers that implement the non-negativity constraints of $\mathcal{M}^+(\edgeset)$.
  Observe that if $\edgeweight^{\eigenfunction}_{uv}=0$, from the first equality in~\eqref{Thm_1st_energy_p2_KKT} necessarily also $c_{uv}$ and $|\grad f(u,v)|$ are zero. 
  In particular, we see that $\edgeweight^f$ necessarily satisfies the following equality:
  \begin{equation}\label{eq_opt_f_edgeweight}
    \edgeweight^{\eigenfunction}_{uv}=
    \frac{|\incidence\eigenfunction (u,v)|^{p-2}}
    {\|\incidence\eigenfunction\|^{2p-4}_{\edgeweight^{\eigenfunction}}}\qquad \forall (u,v)\in \tilde{\edgeset}
  \end{equation}
  Multiplying by $|\incidence\eigenfunction(u,v)|^2$ and summing over the edges we obtain:
  \begin{equation}\label{eq_f_optedgemass}
    \|\incidence\eigenfunction\|_{\edgeweight^{\eigenfunction}}^{2p-2}
    =\|\incidence\eigenfunction\|_p^p\,.
  \end{equation}
  Thus, replacing~\eqref{eq_f_optedgemass} and~\eqref{eq_opt_f_edgeweight} in~\eqref{lower_inf_sup_inequality} we obtain the following lower bound:
  \begin{equation*}
    \begin{aligned}
      \min_{\edgeweight\in\mathcal{M}^+(\edgeset)}
        \lyap_{1,\edgeset}(\edgeweight)\;\geq
      &\max_{\|\eigenfunction\|_{\nodeweight}=1}
      \frac{2p-2}{p}\|\incidence\eigenfunction\|_{p}^{-\frac{p}{p-1}}
      =\frac{2p-2}{p}\eigenval[{[p,2,\nodeweight],1}]^{-\frac{1}{p-1}}\,.
    \end{aligned}
  \end{equation*}
  On the other hand, we can consider the first $[p,2]$-Laplacian eigenvalue $\eigenval[{[p,2,\nodeweight],1}]$ and the corresponding unique eigenfunction $\eigenfunction[{[p,2,\nodeweight],1}]$ with $\|\eigenfunction[{[p,2,\nodeweight],1}]\|_{\nodeweight}=1$, which is strictly positive because of \Cref{p2_first_eigen_charact}.
  Then, consider $   \bar{\edgeweight}$ defined by:
  \begin{equation*}
    \bar{\edgeweight}=\eigenval[{[p,2,\nodeweight],1}]^{\frac{2-p}{p-1}}
    |\incidence\eigenfunction[{[p,2,\nodeweight],1}]|^{p-2}\,.
  \end{equation*}
  The strict positivity of $\eigenfunction[{[p,2,\nodeweight],1}]$ together with  \Cref{Corollary_1st_eigenpair_laplacian} implies that $\eigenfunction[{[p,2,\nodeweight],1}]$ is the first eigenfunction of the $(\bar{\edgeweight},\nodeweight)$-eigenvalue problem with $\eigenval[{(\bar{\edgeweight},\nodeweight),1}]=\eigenval[{[p,2,\nodeweight],1}]^{\frac{1}{p-1}}$.
  Thus we can write:
  \begin{equation*}
    \min_{\edgeweight\in \mathcal{M}^+(\edgeset) }\lyap_{1,\edgeset}(\edgeweight)
    \leq\lyap_{1,\edgeset}(\bar{\edgeweight})
    =\eigenval[{(\bar{\edgeweight},\nodeweight),1}]^{-1}
    +\frac{p-2}{p}
    \eigenval[{[p,2,\nodeweight],1}]^{-\frac{1}{p-1}}
    =\frac{2p-2}{p}
    \eigenval[{[p,2,\nodeweight],1}]^{-\frac{1}{p-1}}\,,
  \end{equation*}
  which concludes the proof.
\end{proof}

Since, as mentioned before, the $[p,2]$ eigenvalue problem is of independent interest, before going back to the classical $p$-Laplacian eigenvalue problem we conclude this section by noting that, given a fixed density $\nodeweight$ on the internal nodes, the class of energy functions
\begin{equation*}
  \lyap_{k,\edgeset}(\edgeweight)=
  \frac{1}{\eigenval[{(\edgeweight,\nodeweight)},k]}
  +\mass_{\edgeset,p}(\edgeweight)
\end{equation*}
can be used to characterize  $[p,2]$-Laplacian eigenpairs,
in analogy with the $(\optedgeweight,\optnodeweight)$ case of Theorem~\ref{thm:Smooth_saddle_points_correspond_to_p_eigenpairs}. We collect this result in the following Theorem that can be proved analogously to \Cref{thm:Smooth_saddle_points_correspond_to_p_eigenpairs}.

\begin{theorem} 
  Let $\optedgeweight\in\mathcal{M}^+(\edgeset)$ be a differentiable minimizer of the energy function $\lyap_{k,\edgeset}(\edgeweight)$. Then, $(\eigenval[{(\optedgeweight,\nodeweight),k}]^{p-1},\eigenfunction[{(\optedgeweight,\nodeweight),k}])$ is a $[p,2]$-Laplacian eigenpair.
\end{theorem}




\subsection{From the $[p,2]$-Laplacian to the $p$-Laplacian Eigenvalue Problem}

This paragraph is dedicated to the proof of Theorem~\ref{Thm_saddle_point_1st_pp_eigenpair}. To this aim, we start by observing that, analogously to the equivalence of the $p$-Laplacian eigenvalue problem with a generalized linear eigenvalue problem (eq.~\eqref{weighted_lap_eq}), a pair $(\eigenval,\eigenfunction)$ is an eigenpair of the $p$-Laplacian operator if and only if it satisfies the following constrained weighted $[p,2]$-Laplacian eigenvalue problem:
\begin{equation*}
  \begin{cases}
    \plap\eigenfunction(u)
    =\eigenval\nodeweight_u
    \|\eigenfunction\|_{2,\nodeweight}^{p-2} \eigenfunction(u) \quad
    &\forall u\in\internalnodes \\[5pt]
    \nodeweight_u=
    |\eigenfunction(u)|^{p-2}/    \|\eigenfunction\|_{2,\nodeweight}^{p-2}
  &\forall u\in\internalnodes
  \end{cases}\,.
\end{equation*}
%
In Section~\ref{p,2_Laplacian Section} we have proved that, given a nonsingular weight function $\nodeweight$ on the nodes, it is possible to characterize the first eigenpair of the $[p,2]$-Laplacian eigenvalue problem weighted in $\nodeweight$ by the minimizer $\optedgeweight_{\nodeweight}$ of the function $\lyap[1,\edgeset](\edgeweight)$ (see Theorem~\ref{Thm_p-2_Lyap_Minimizer}). Similarly to what done before, here we introduce an energy function depending only on the variable $\nodeweight$ and given by:
\begin{equation*}
  \lyap[1,\nodeset](\nodeweight)
  =\frac{2(p-1)}{p}
  \eigenval[{[p,2,\nodeweight],1}]^{-\frac{1}{p-1}}
  -\frac{p-2}{p}\sum_{u\in\internalnodes}\nodeweight(u)^{\frac{p}{p-2}}\,.
\end{equation*}
Observe that for any $\nodeweight\neq 0$, from Theorem \ref{Thm_p-2_Lyap_Minimizer}, we have the following equality:
\begin{equation*}
  \lyap[1,\nodeset](\nodeweight)
  =\lyap[1,\edgeset](\optedgeweight_{\nodeweight})
  -{\mass}_{V,p}(\nodeweight)
  =\Eps_{p,1}(\optedgeweight_{\nodeweight},\nodeweight),
\end{equation*}
Moreover, since $\rayl_{p,2,0}^{-1}(f)=0$ for any $f\neq 0$, $\lyap[1,\nodeset]$ can be extended to zero by setting $\lyap[1,\nodeset](0):=0$.
Now we want to show that there exists a unique critical point of $\lyap[1,\nodeset]$ and that this critical point corresponds to the unique first eigenpair of the $p$-Laplacian operator.
We start our goal by collecting some preliminary results needed in the proofs. First, in the next Lemma we address the differentiability of the function $\nodeweight\mapsto\eigenval[{[p,2,\nodeweight],1}]$. Note that similar results are available in the continuous case for the regularity of the first $p$-Laplacian eigenfunction with respect to perturbations of the domain~\cite{lamberti2003differentiability}. We move the technical proof to \cref{sec:appendix}.

\begin{lemma}\label{lemma_1st_p2_eigen_gradient}
  Let $\eigenval[1]$ and $\eigenfunction[1]$ be the minimum value and the minimum point of  $\rayl_{p,2,\nodeweight}(\eigenfunction)$. Then the function  $\eigenval[1]:\nodeweight\mapsto\eigenval[{[p,2,\nodeweight],1}]$ and its first derivatives are continuous, i.e., $\eigenval[1]\in C^1(\mathcal{M}^+(\internalnodes)\setminus\{0\},\R)$. Moreover:
  \begin{equation*}
    \frac{\partial\eigenval[1]}{\partial\nodeweight}(\nodeweight_0)
    =-\frac{p}{2}
    \frac{\eigenval[1]|\eigenfunction_{[p,2,\nodeweight_0],1}|^2}
    {\|\eigenfunction_{[p,2,\nodeweight_0],1}\|_{2,\nodeweight_0}^2}\,.
  \end{equation*}
\end{lemma}

The next theorem asserts that there exists a unique maximum point $\optnodeweight$ of the function $\lyap[1,\nodeset](\nodeweight)$, which is everywhere nonzero and it identifies the unique first eigenpair of the $p$-Laplacian. In particular we write $\mathrm{Int(\mathcal{M}^+(\internalnodes))}=
      \{\nodeweight:\internalnodes\rightarrow\R\,|\;
      \nodeweight_u>0\;\:\forall u\in\internalnodes\}\,$ to denote the interior of $\mathcal{M}^+(\internalnodes)$.

\begin{theorem}\label{Thm_1st_pp_characterization_as_crtcl_point}
  The function $\lyap[1,\nodeset](\nodeweight)$ admits a unique maximum point $\optnodeweight\in \mathcal{M}^+(\internalnodes)$ that satisfies the following properties:
  
    %
    %
    \begin{enumerate}
    \item $\optnodeweight\in \mathrm{Int}\big(\mathcal{M}^+(\internalnodes)\big)$, i.e., $\nodeweight_u>0$ for all $u\in\internalnodes$.
  \item The first eigenpair  $\big(\eigenval[{[p,2,\optnodeweight],1}], \eigenfunction[{[p,2,\optnodeweight],1}]\big)$ of the weighted $[p,2,\optnodeweight]$-Laplacian is related to the first eigenpair of the $[p,p]$-Laplacian by:
    \begin{equation*}
      \left(\eigenval[{[p,2,\optnodeweight],1}]^{\frac{p}{2(p-1)}},
        \eigenfunction[{[p,2,\optnodeweight],1}]\right)=
      \left(\eigenval[{[p,p],1}],\eigenfunction[{[p,p],1}]\right)\quad 
      \mbox{and} \quad
      \lyap_{1,\nodeset}(\optnodeweight)=\eigenval[{[p,p],1}]^{\frac{2}{p}}\,.
    \end{equation*}
  \item No other internal critical points of the function $\lyap[1,\nodeset](\nodeweight)$ exist.
  \end{enumerate}
\end{theorem}

\begin{proof}
  Observe that the first nonzero eigenvalue of the $[p,2,\nodeweight]$-Laplacian is given by  $\eigenval[{[p,2,\nodeweight],1}]=\min_{f\neq 0}\|\incidence f\|_p^p/\|f\|_{2,\nodeweight}^p$,
  where the $[p,2]$-Rayleigh quotient, admitting it could take values in $[0,\infty]$, is always well defined (see Theorem \ref{p2_first_eigen_charact}).
  Hence we can write:
  \begin{equation}\label{eq_exchange_maxima}
    \max_{\nodeweight\in\mathcal{M}^+(\internalnodes)}\lyap_{1,\nodeset}=
    \max_{f\neq 0}\max_{\nodeweight\in\mathcal{M}^+(\internalnodes)}
    \frac{2p-2}{p}
    \left(\frac{\|f\|_{2,\nodeweight}^p}
      {\|\incidence f\|_{p}^p}\right)^{\frac{1}{p-1}}- \frac{p-2}{p}\sum_{u\in\internalnodes}\nodeweight_u^{\frac{p}{p-2}}\,.
  \end{equation}
  Now, assume $f$ to be fixed and $\nodeweight^f$ realizing the maximum below:
  \begin{equation*}
    \nodeweight^f\in \argmax_{\nodeweight\in\mathcal{M}^+(\internalnodes)}
    \frac{2p-2}{p}
    \left(
      \frac{\|f\|_{2,\nodeweight}^p}{\|\incidence f\|_{p}^p}
    \right)^{\frac{1}{p-1}}
    -\frac{p-2}{p}\sum_{u\in\internalnodes}\nodeweight_u^{\frac{p}{p-2}}\,.
  \end{equation*}
  Since the last is a constrained maximum problem, by the KKT conditions, there exist a family of Lagrange multipliers $\{c_u\}_{u\in\internalnodes}$ such that:
  \begin{equation*}
    \begin{cases}
      \displaystyle{
      \rayl_{p,2,\nodeweight^f}^{-\frac{1}{p-1}}(f)
      \frac{|f(u)|^2}{\|f\|_{2,\nodeweight^f}^2}
      -{(\nodeweight^f_u)}^{\frac{2}{p-2}}+c_u=0}
      \quad &\forall u\in\internalnodes\\[6pt]
      c_u\nodeweight^f_u=0 \quad \text{and} \quad  c_u\geq 0
            &\forall u\in\internalnodes
    \end{cases}\,.
  \end{equation*}
  In particular, whenever $\nodeweight_u=0$ necessarily also $f(u)=c(u)=0$, yielding the following equality:
  \begin{equation}\label{pp_Thrm_optnodeweight_eq_2}
    \nodeweight^f_u
    =\left(\rayl_{p,2,\nodeweight^f}(f)\right)^{-\frac{p-2}{2(p-1)}}
    \frac{|f(u)|^{p-2}}
    {\|f\|_{2,\nodeweight^f}^{p-2}} \qquad \forall u\in\internalnodes\,.
  \end{equation}
  Multiplying by $|f(u)|^2$ in \eqref{pp_Thrm_optnodeweight_eq_2} and summing over $u\in\internalnodes$, the $(2,\nodeweight^f)$ seminorm of $f$ can be written as
  \begin{equation}\label{eq_norm_equality}
  \|f\|_{2,\nodeweight^f}^p
  =\left(\rayl_{p,2,\nodeweight^f}(f)\right)^{-\frac{p-2}{2(p-1)}}\|f\|_p^p
  =\|f\|_p^{2p-2}/\|\incidence f\|_p^{p-2}.
  \end{equation}
  In particular, exploiting the expressions \eqref{pp_Thrm_optnodeweight_eq_2} and \eqref{eq_norm_equality} we can derive the following expression for the $p/(p-2)$-norm of $\nodeweight^f$:
  \begin{equation}\label{eq_nuf_mass}
    \sum_{u\in\internalnodes}(\nodeweight^f_u)^{\frac{p}{p-2}}
    =\frac{\|f\|_p^2}{\|\incidence f\|_p^2}
  \end{equation}
  Finally if we replace the expressions from \eqref{eq_norm_equality} and \eqref{eq_nuf_mass} in \eqref{eq_exchange_maxima}, we can now calculate the maximum of $\lyap_{1,\nodeset}$:
  \begin{equation*}
  \begin{aligned}
    \max_{\nodeweight\in \mathcal{M}^+(\internalnodes)}\lyap_{1,\nodeset}
    &=\max_{f\neq 0}\frac{2p-2}{p}
    \frac{\|f\|_{p}^{\frac{2p-2}{p-1}}}{\|\incidence f\|_{p}^{\frac{2p-2}{p-1}}}
    -\frac{p-2}{p}\frac{\|f\|_p^2}{\|\incidence f\|_p^2}=\max_{f\neq 0}\frac{\|f\|_{p}^{2}}{\|\incidence f\|_p^2}
    =\eigenval[{[p,p],1}]^{-\frac{2}{p}}\,
    \end{aligned}
  \end{equation*}
  and, since the $1$-st $p$-Laplacian eigenfunction $f_{[p,p],1}$ realizes the maximum in $f$, from \eqref{pp_Thrm_optnodeweight_eq_2} the maximizer $\optnodeweight$ satisfies:
  \begin{equation*}
    \optnodeweight
    =
    \eigenval[{[p,p],1}]^{-\frac{2(p-2)}{p^2}}
    \frac{|\eigenfunction[{[p,p],1}]|^{p-2}}
    {\|\eigenfunction[{[p,p],1}]\|_p^{p-2}}\,.
  \end{equation*}
  In addition, we know that $\eigenfunction[{[p,p],1}](u)>0$ for any $u\in\internalnodes$ (see Theorem \ref{Thm:Characterization_of_the_first_eigenvalues}).
  Thus $\optnodeweight\in\mathrm{Int}(\mathcal{M}^+(\internalnodes))$ and it is the unique maximizer.
  
  To conclude the proof, we observe that if $\nodeweight$ is a critical point of $\lyap_{1,\nodeset}$ with $\nodeweight\in\mathrm{Int}(\mathcal{M}^+(\internalnodes))$, then, from Lemma~ \ref{lemma_1st_p2_eigen_gradient}, we have
  \begin{equation*}
    \begin{cases}
      \displaystyle{
        \eigenval[{[p,2,\nodeweight],1}]^{-\frac{1}{p-1}}
        \frac{|\eigenfunction[{[p,2,\nodeweight],1}](u)|^2}
        {\|\eigenfunction[{[p,2,\nodeweight],1}]\|_{2,\nodeweight}^2}
        -{\nodeweight(u)}^{\frac{2}{p-2}}}=0
      \quad &\forall u\in\internalnodes\\[6pt]
      \plap\eigenfunction[{[p,2,\nodeweight],1}]=
      \eigenval_1(p,2,\nodeweight)
      \|\eigenfunction[{[p,2,\nodeweight],1}]\|_{2,\nodeweight}^{p-2}
      \nodeweight\odot\eigenfunction[{[p,2,\nodeweight],1}]
    \end{cases}\,,
  \end{equation*}
  i.e. $\plap\eigenfunction[{[p,2,\nodeweight],1}] =
    \eigenval[{[p,2,\nodeweight],1}]^{\frac{p}{2(p-1)}}
    |\eigenfunction[{[p,2,\nodeweight],1}]|^{p-2}\odot
    \eigenfunction[{[p,2,\nodeweight],1}]$.  %
  But then, since $\eigenfunction[{[p,2,\nodeweight],1}]$ is the first $[p,2]$-Laplacian eigenfunction, Theorem~\ref{p2_first_eigen_charact} ensures that $\eigenfunction[{[p,2,\nodeweight],1}](u)> 0$ for all $u\in\internalnodes$, and thus  $\eigenfunction[{[p,2,\nodeweight],1}]=\eigenfunction[{[p,p],1}]$, i.e. $\nodeweight=\optnodeweight$.
\end{proof}

These results lead directly to the proof of Theorem~\ref{Thm_saddle_point_1st_pp_eigenpair}\,.
\begin{proof}[Proof of Theorem~\ref{Thm_saddle_point_1st_pp_eigenpair}]
Theorems~\ref{Thm_1st_pp_characterization_as_crtcl_point}   and~\ref{Thm_p-2_Lyap_Minimizer} ensure that there exists a unique $(\optedgeweight,\optnodeweight=
    \argmax_{\nodeweight\in \mathcal{M}^+(\internalnodes)\setminus{0}}\;
    \argmin_{\edgeweight\in \mathcal{M}^+(\edgeweight)}=
    1/\eigenval_1(\edgeweight,\nodeweight)
    +\mass_{\edgeset,p}(\edgeweight)-\mass_{\nodeset,p}(\nodeweight)\,.$.
  In particular, $(\optedgeweight,\optnodeweight)$ is the only, possibly non-differentiable, saddle point of the function $\Eps_{p,1}(\edgeweight,\nodeweight).$
  In addition, since $\optedgeweight\in\argmax_{\edgeweight}1/\eigenval[{(\edgeweight,\nodeweight^*),1}]+{\mass}_{\edgeset,p}(\edgeweight)$, Theorem  \ref{Thm_p-2_Lyap_Minimizer} implies the existence of a first eigenpair $(\eigenfunction_{(\optedgeweight,\optnodeweight),1},\lambda_{(\optedgeweight,\optnodeweight),1})$ of the $(\optedgeweight,\optnodeweight)$ eigenvalue problem \eqref{linear_weighted_eigenpairs} satisfying the equality $(\eigenfunction_{(\optedgeweight,\optnodeweight),1},
      \lambda_{(\optedgeweight,\optnodeweight),1}^{p-1})
    =(\eigenfunction_{[p,2,\optedgeweight],1},
      \lambda_{[p,2,\optedgeweight],1})$.
  Finally, as a consequence of Theorem \ref{Thm_1st_pp_characterization_as_crtcl_point}, we have the equality $(\eigenfunction_{[p,2,\optedgeweight],1},
      \lambda_{[p,2,\optedgeweight],1}^{p/2(p-1)})
    =(\eigenfunction_{[p,p],1},
      \lambda_{[p,p],1})$,  which concludes the proof.
\end{proof}
%
%
%
%
%
%
%
%
\subsection{Numerical evaluation of the saddle points}

We have observed that every $p$-Laplacian eigenpair can be considered as a linear eigenpair of a properly weighted Laplacian eigenproblem. This characterization allowed us to introduce a class of energy functions whose differentiable saddle points correspond to $p$-Laplacian eigenpairs.
Now it is thus natural to investigate numerical methods for the computation of $p$-Laplacian eigenpairs based on gradient flows of the functions $\Eps_{p,k}(\edgeweight,\nodeweight)$ .
Next, we present some preliminary numerical results showing that these schemes actually deliver acceptable results in most situations. Nevertheless  the problem of the lack of regularity of the functions $\Eps_{p,k}(\edgeweight,\nodeweight)$ in case of eigenvalues with multiplicity greater than one is still a stumbling block that prevent the convergence of the numerical schemes in some situations. These are evidenced by bounded oscillations of residuals and non-convergence of the algorithm in some cases. In particular, we believe that a theoretical investigation of the numerical schemes is necessary to fully understand  some of these problems and possibly overcome them. We propose this as a future research direction.
The computation of the saddle points of the energy functions $\Eps_{p,k}(\edgeweight,\nodeweight)$ is a constrained critical point problem. To incorporate in our formulation the positivity constraint we follow the same procedure that turned out to be successful in the solution of the $L^1$-Optimal Transport problem and discussed in~\cite{facca3,piazzon2023computing}, indeed our energy functions were inspired by the $L^1$-optimal transport energy functions therein presented and studied. We perform the following change of variable, which preserves the positivity of $\edgeweight$ and $\nodeweight$: $\edgeweight=\sigma_1^{2(p-2)/p}$, $\nodeweight=\sigma_2^{2(p-2)/p}$.
Using the new variables, the energy function $\Eps_{p,k}(\sigma_1^{{2(p-2)/p}},\sigma_2^{{2(p-2)/p}})$ becomes well defined everywhere in $\R^{|\edgeset|}\times\R^{|\nodeset|}$ except in the points $(\edgeweight, \nodeweight)$ where $\dim(\ker(\lap[\edgeweight])\cap \ker(\diag(\nodeweight)))^{\perp}$ is smaller than $k$. We thus define a dynamics for the variables $(\edgeweight,\nodeweight)$ as the gradient flow with respect to the variables $\sigma_1$ and $\sigma_2$. To this aim, for $p>2$, we use the following time-derivatives, derived analogously one to each other:
\begin{equation*}
  \begin{aligned}
    \dot{\edgeweight}&=2\frac{(p-2)}{p}\sigma_1^{\frac{p-4}{p}}\dot\sigma_1
    =-2\frac{(p-2)}{p}\sigma_1^{\frac{p-4}{p}}
    \frac{\partial}{\partial\sigma_1}
    \big(\Eps_{p,k}(\sigma_1^{\frac{2(p-2)}{p}},\sigma_2^{\frac{2(p-2)}{p}})\big)%
    \\
    &=-4\frac{(p-2)^2}{p^2}\sigma_1^{\frac{2(p-4)}{p}}
    \frac{\partial}{\partial\edgeweight}\Eps_{p,k}(\edgeweight,\nodeweight)
    =-4\frac{(p-2)^2}{p^2}\edgeweight^{\frac{p-4}{p-2}}
    \frac{\partial}{\partial \edgeweight}\Eps_{p,k}(\edgeweight,\nodeweight)\\[1em]
 \dot{\nodeweight}&=
  4\frac{(p-2)^2}{p^2}\nodeweight^{\frac{p-4}{p-2}}
  \frac{\partial}{\partial\nodeweight}\Eps_{p,k}(\edgeweight,\nodeweight)\,.
  \end{aligned}
\end{equation*}
Writing explicitly the partial derivatives and neglecting constant multiplicative factors, which turn out to be just a variation of the speed of the dynamics, we end up with the following gradient flow system:
\begin{equation*}
  \begin{cases}
    \dot{\edgeweight}
    =\edgeweight^{\frac{p-4}{p-2}}\displaystyle{\big(|\incidence\eigenfunction[{(\edgeweight,\nodeweight),k}]|^2\big/\eigenval[{(\edgeweight,\nodeweight),k}]\|\eigenfunction[{(\edgeweight,\nodeweight),k}]\|^2_{\nodeweight}\big)}-\edgeweight\\[.8em]
    \dot{\nodeweight}=\nodeweight^{\frac{p-4}{p-2}}\displaystyle{\big(|\eigenfunction[{(\edgeweight,\nodeweight),k}]|^2\big/\|\incidence
    \eigenfunction[{(\edgeweight,\nodeweight),k}]\|^2_{\edgeweight}\big)}
    -\nodeweight\\
    \lap[\edgeweight+\delta]\eigenfunction[{(\edgeweight,\nodeweight),k}]=
    \eigenval[{(\edgeweight,\nodeweight),k}]\diag(\nodeweight+\delta)
    \eigenfunction[{[\edgeweight,\nodeweight],k}]
  \end{cases}\,,
\end{equation*}
where a small regularization parameter $\delta>0$ was included to enforce the condition $\ker(\lap[\edgeweight])\cap \ker(\diag(\nodeweight))=0.$
The equilibrium point of the above dynamics provides, up to $\delta$, an approximation of the sought eigenpair.
The first two algebraic-differential equations are discretized by means of a simple explicit Euler method with an empirically-determined and constant time step size $\tau$.
The third equation is a purely algebraic linear eigenvalue problem of the $\edgeweight$-weighted linear Laplacian solved by means of standard Lapack routines~\cite{lapack99}. No effort has been done to exploit sparsity of the the graph-related matrices, which could provide important computational efficiency improvements. Thus, when looking for the $k$-th eigenpair starting from given initial values $\edgeweight^0=\edgeweight_k^0$ and $\nodeweight^0=\nodeweight_k^0$, the $n=1,2,\ldots$ approximations are calculated with the following iterative scheme:
\begin{itemize}
\item[--] \textit{choose $\delta>0$
    and compute $(\eigenval^{n+1},\eigenfunction^{n+1})$ solving:}
  \begin{equation*} \lap[\edgeweight^n+\delta]\eigenfunction^{n+1}
    =\eigenval[{(\edgeweight^n,\nodeweight^n),k}]^{n+1}
    \diag(\nodeweight^n+\delta)\eigenfunction^{n+1}
  \end{equation*}
  \item[--] \textit{compute $(\edgeweight^{n+1}, \nodeweight^{n+1})$ by:}
    \begin{equation*}
      \begin{aligned}
        &\edgeweight^{n+1}
          :=\edgeweight^n+\tau
          \big(
            (\edgeweight^n)^{\frac{p-4}{p-2}}\frac{|\incidence\eigenfunction^{n+1}|^2}
            {(\eigenval^{n+1})^2\|\eigenfunction^{n+1}\|^2_{\nodeweight^n}}
            -\edgeweight^n
          \big)\\[.6em]
        &\nodeweight^{n+1}
          :=\nodeweight^{n}+\tau\big(
            (\nodeweight^n)^{\frac{p-4}{p-2}}
            \frac{|\eigenfunction^{n+1}|^2}%
            {\|\incidence\eigenfunction^{n+1}\|_{\edgeweight^n}^2}
            -\nodeweight^n\big).     
      \end{aligned}
    \end{equation*}
\end{itemize}
It is possible to see that the previous iteration preserves the positivity constraint of $\edgeweight$ and $\nodeweight$ provided $(\edgeweight^0,\nodeweight^0)>0$ and $\tau\leq1$.
We apply the above scheme for the calculation of the first nine eigenpairs of the $\plap$ operator with $p=3$ on a unit square graph. The graph is  a uniform $21\times21$ discretization with edge weights given by the reciprocal of the edge lengths.
Homogeneous Dirichlet boundary conditions are imposed at the boundary chosen as the nodes lying on the sides of the unit square. 
The choice of the square graph here is simply driven by the possibility to visualize and interpret the computed eigenfunctions, nontheless experiments have been performed on random and less structured graphs delivering similar results.

The regularization parameter is set to $\delta=10^{-8}$ and the time step is fixed at $\tau=10^{-1}$.
Convergence towards equilibrium is considered achieved when $\mathrm{err}:=\max\left\{\mathrm{err}_{\edgeweight},\mathrm{err}_{\nodeweight}\right\}$ is below a given tolerance, where
\begin{equation*}
  \begin{gathered}
    \mathrm{err}_{\edgeweight}
    :=\frac{\|\edgeweight^{n+1}-\edgeweight^{n}\|_2}{\tau\|\edgeweight^{n}\|_2} 
    \quad \text{ and } \quad 
    \mathrm{err}_{\nodeweight}
    :=\frac{\|\nodeweight^{n+1}-\nodeweight^{n}\|_2}{\tau\|\nodeweight^{n}\|_2}\,.
  \end{gathered}
\end{equation*}
The accuracy of the computed eigenpair is verified by looking to the residual defined as
\begin{equation}\label{residual}
  \mathrm{res}=\frac{
    \|\plap\eigenfunction^{n+1}
    -(\eigenval^{n+1})^{\frac{p}{2}}
    |\eigenfunction^{n+1}|^{p-2}\eigenfunction^{n+1}\|_{2}}%
  {\|(\eigenval^{n+1})^{\frac{p}{2}}
    |\eigenfunction^{n+1}|^{p-2}\eigenfunction^{n+1}\|_2}.
\end{equation}
%
%
%

\begin{figure}  
  \begin{center}
    \begin{minipage}{.32\textwidth}      
      \includegraphics[width=\textwidth]{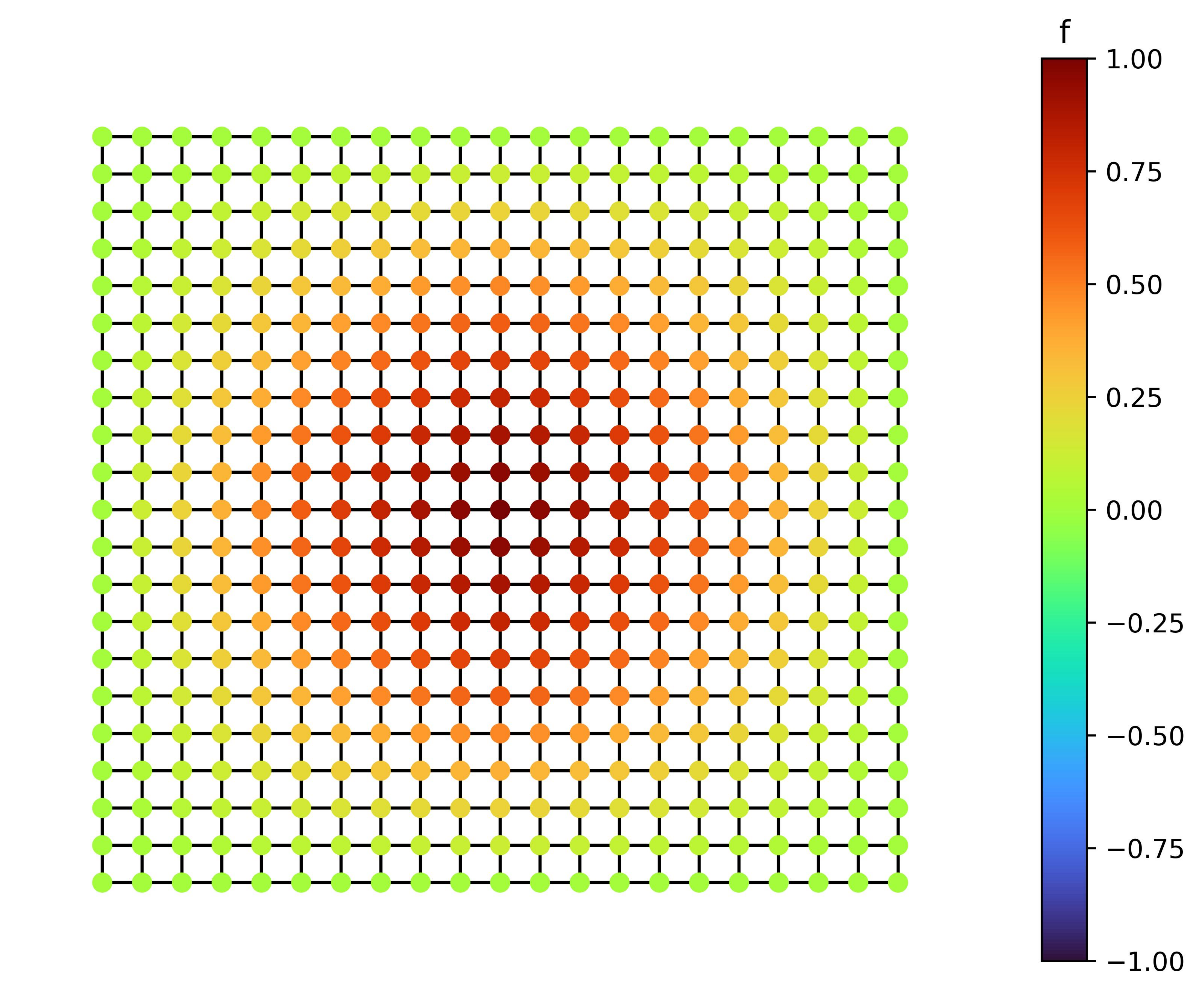}
      \includegraphics[width=\textwidth,height=2.5cm]{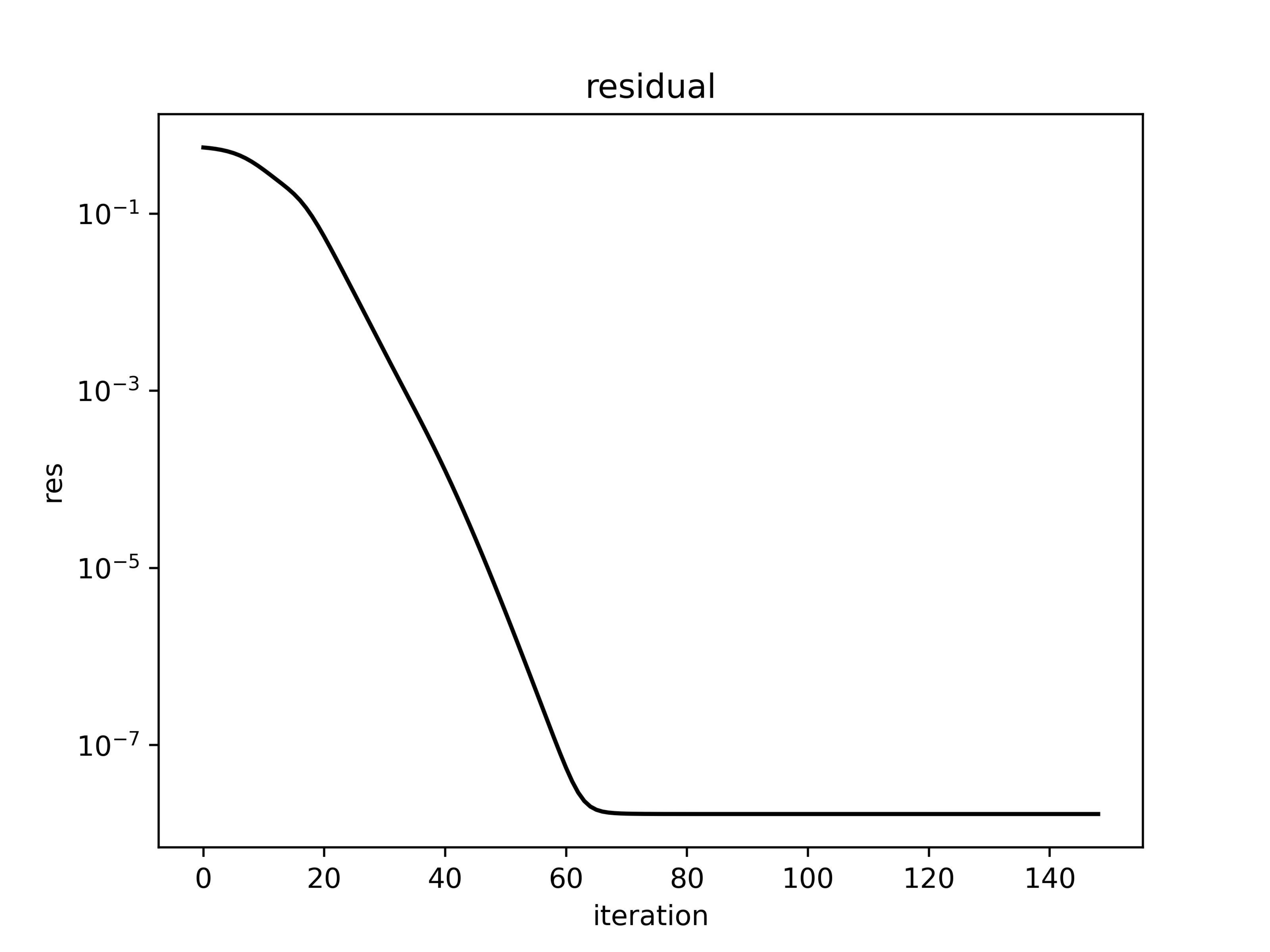}
    \end{minipage}
    \begin{minipage}{.32\textwidth}       
      \includegraphics[width=\textwidth]{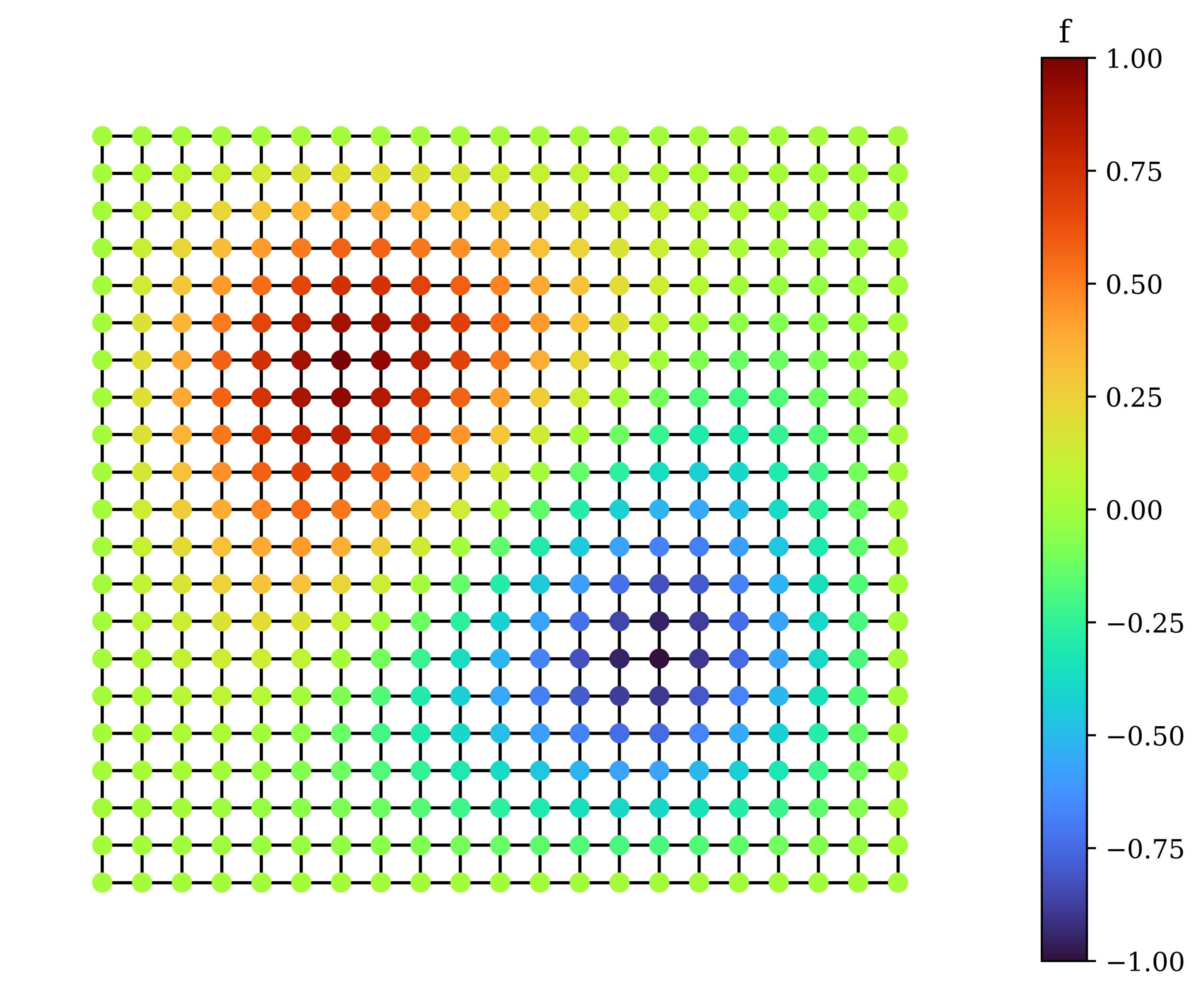}
      \includegraphics[width=\textwidth,height=2.5cm]{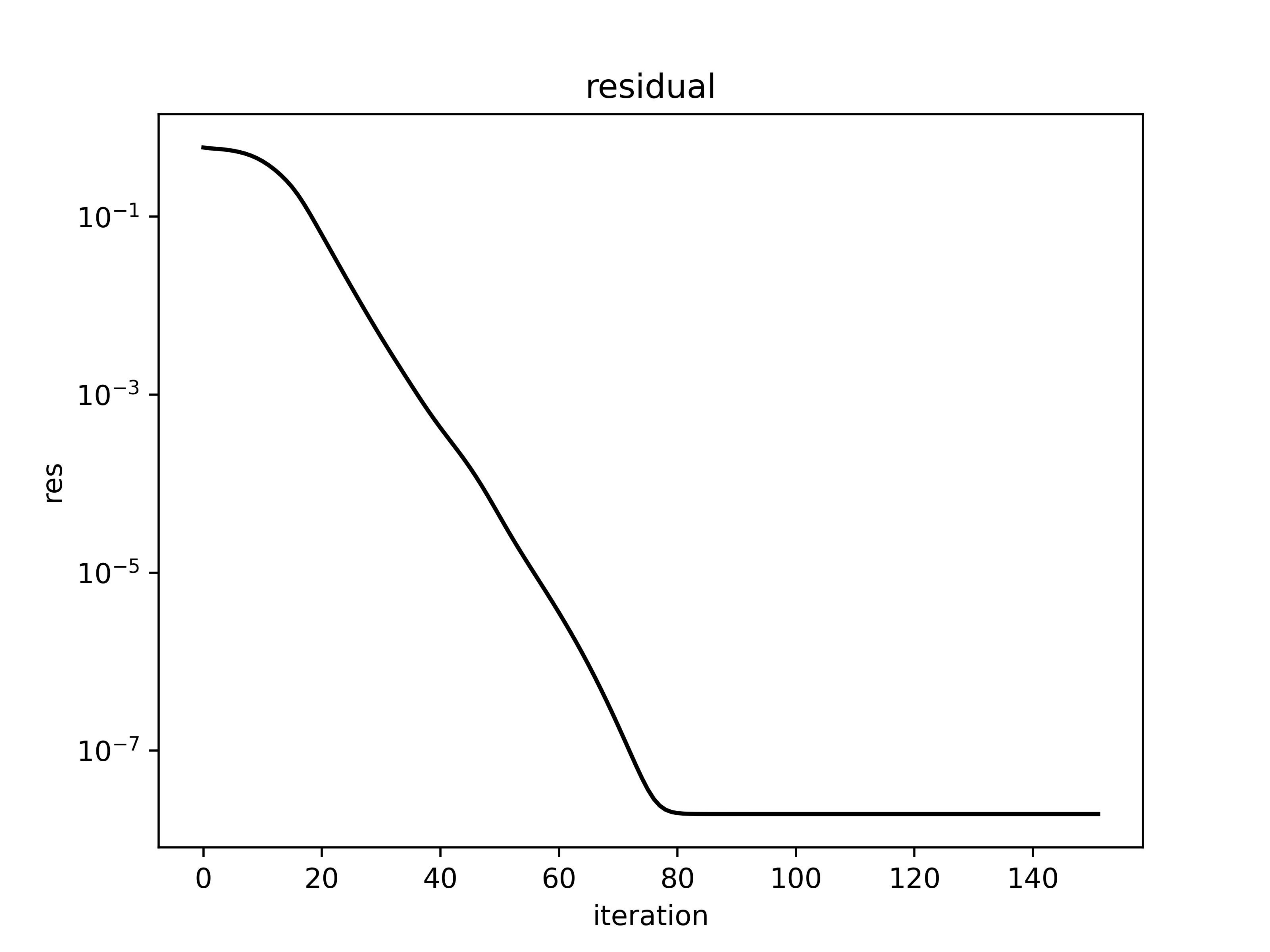}
    \end{minipage}
    \begin{minipage}{.32\textwidth}              
      \includegraphics[width=\textwidth]{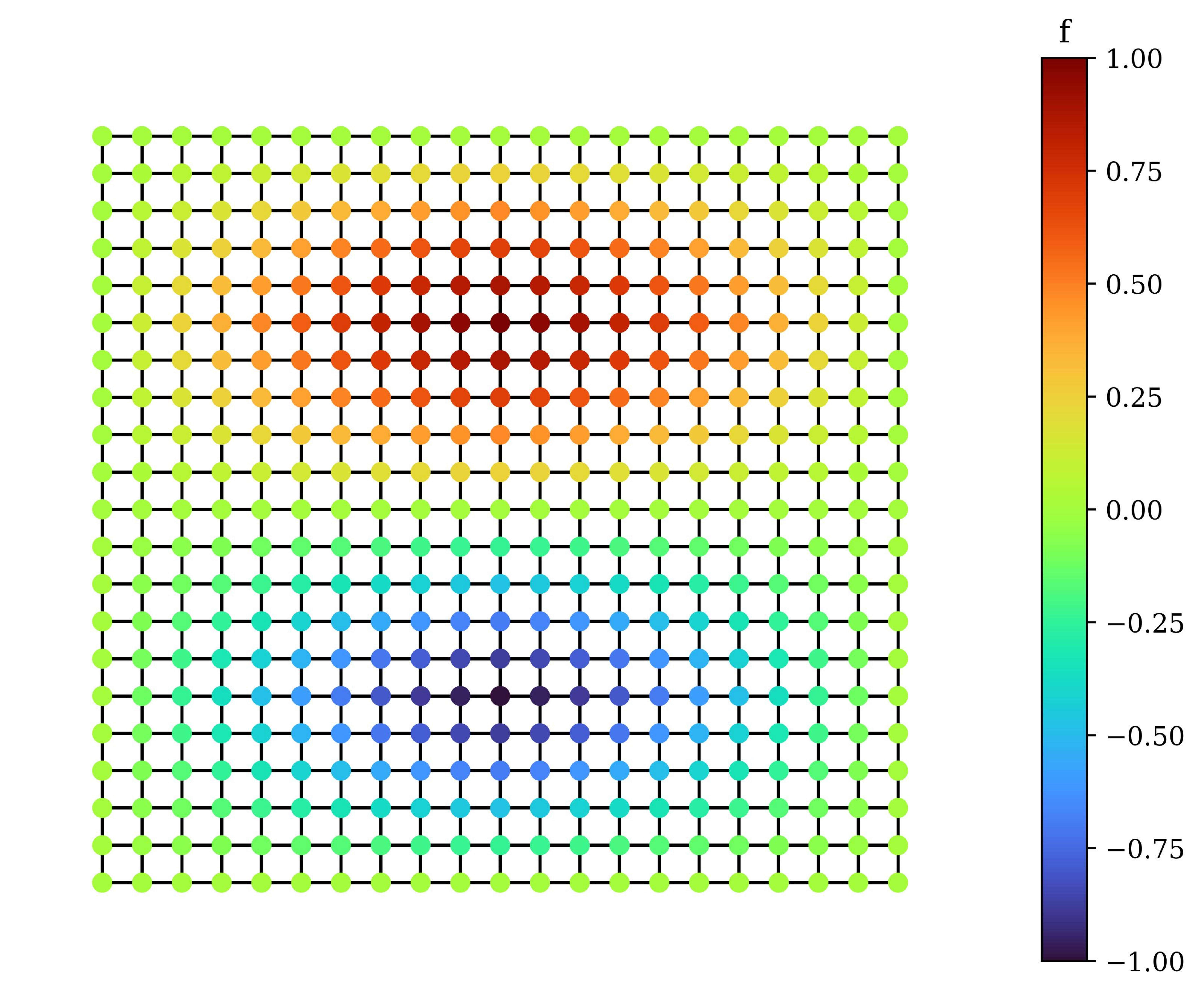}
      \includegraphics[width=\textwidth,height=2.5cm]{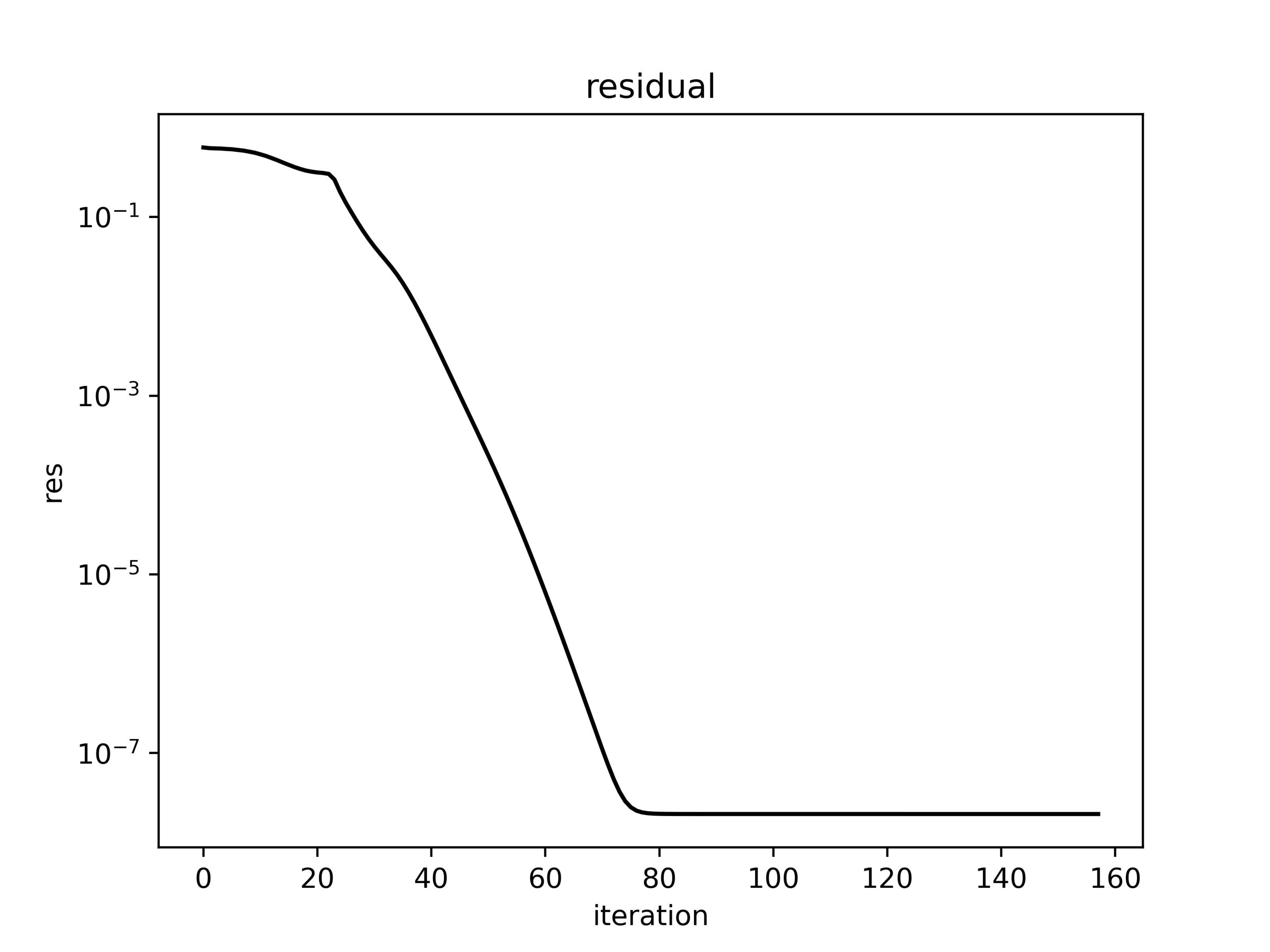}
    \end{minipage}
    \begin{minipage}{.32\textwidth}       
      \includegraphics[width=\textwidth]{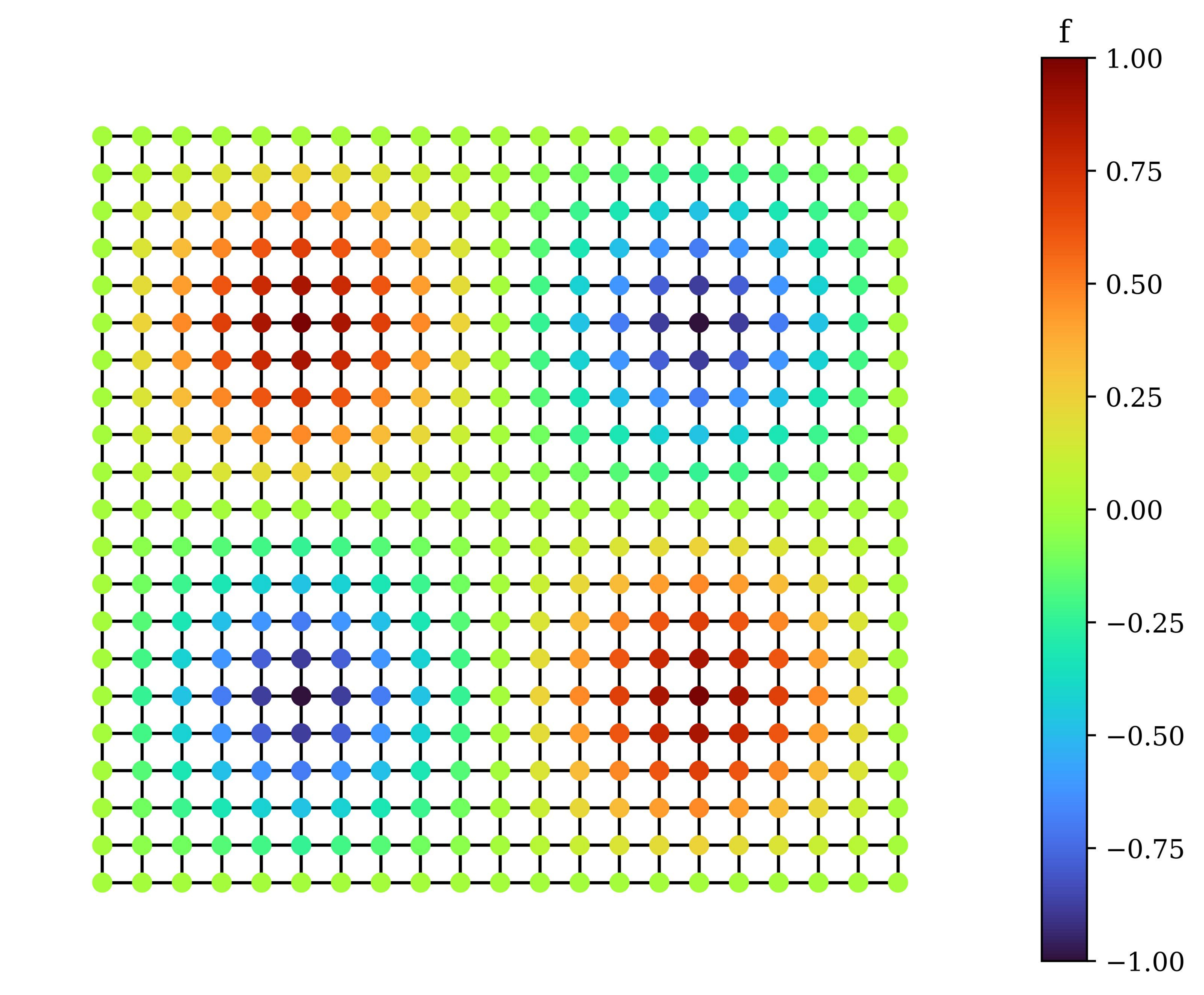}
      \includegraphics[width=\textwidth,height=2.5cm]{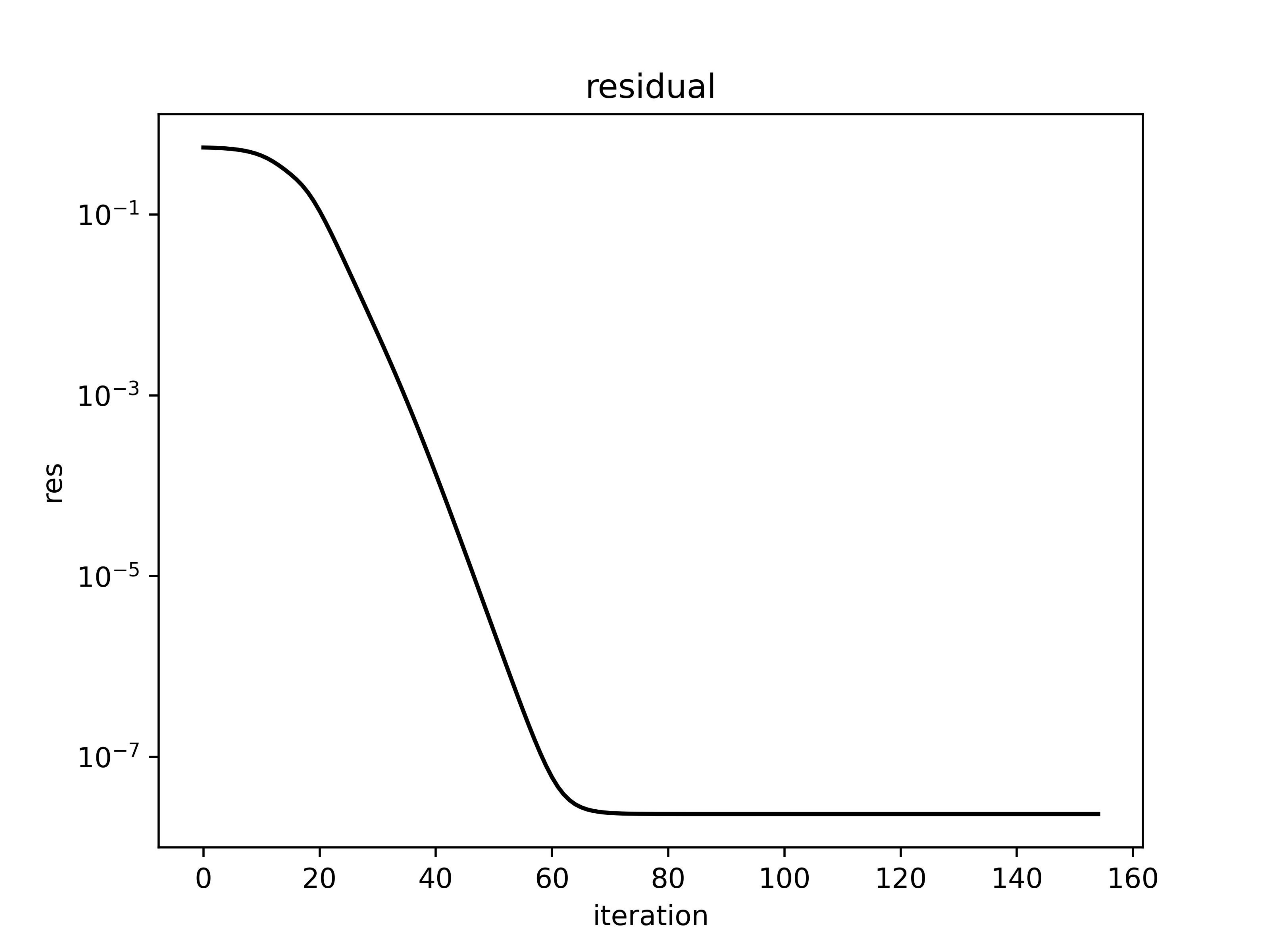}
    \end{minipage}
    \begin{minipage}{.32\textwidth}       
      \includegraphics[width=\textwidth]{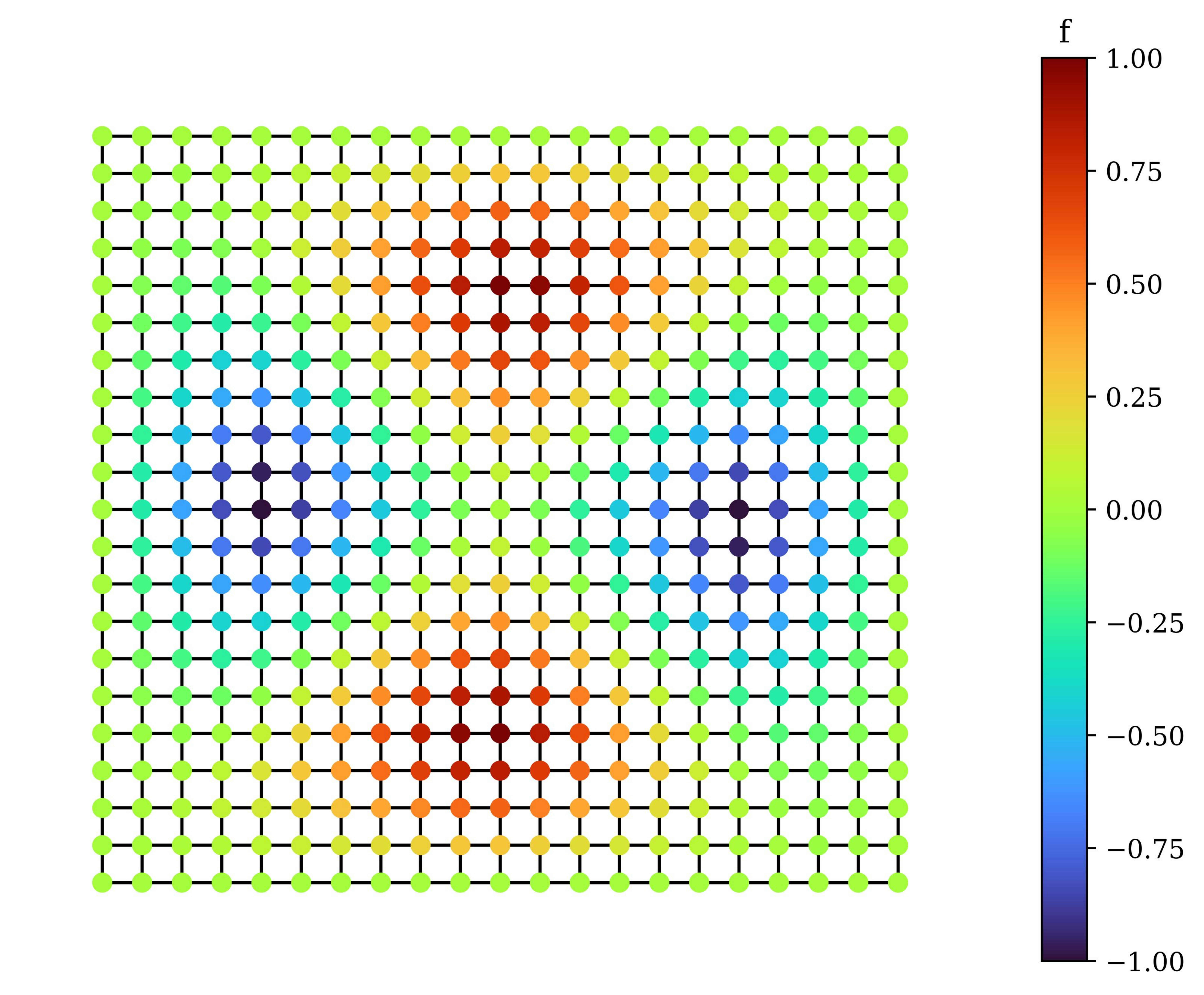}
      \includegraphics[width=\textwidth,height=2.5cm]{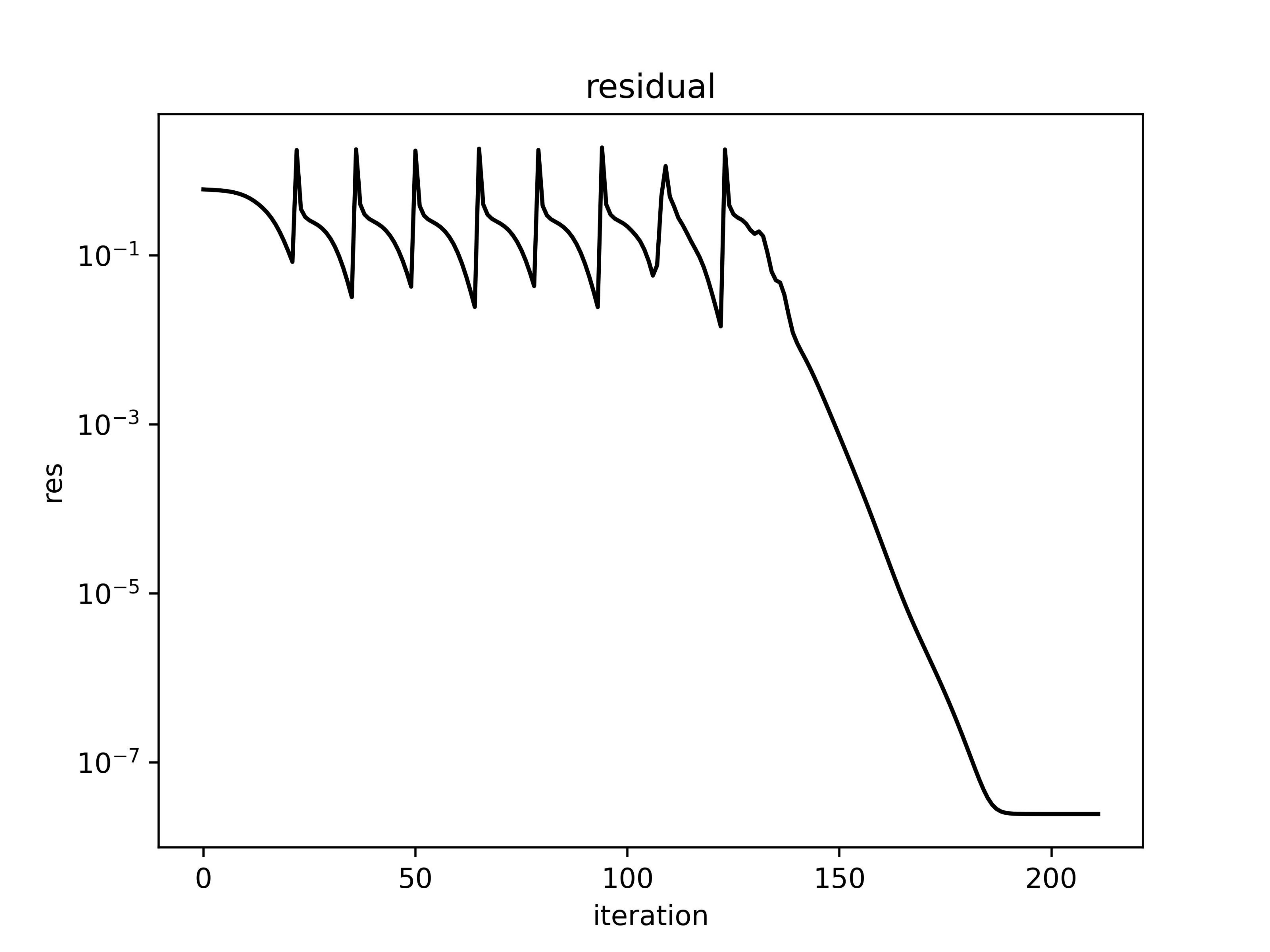}
    \end{minipage}
    \begin{minipage}{.32\textwidth}       
      \includegraphics[width=\textwidth]{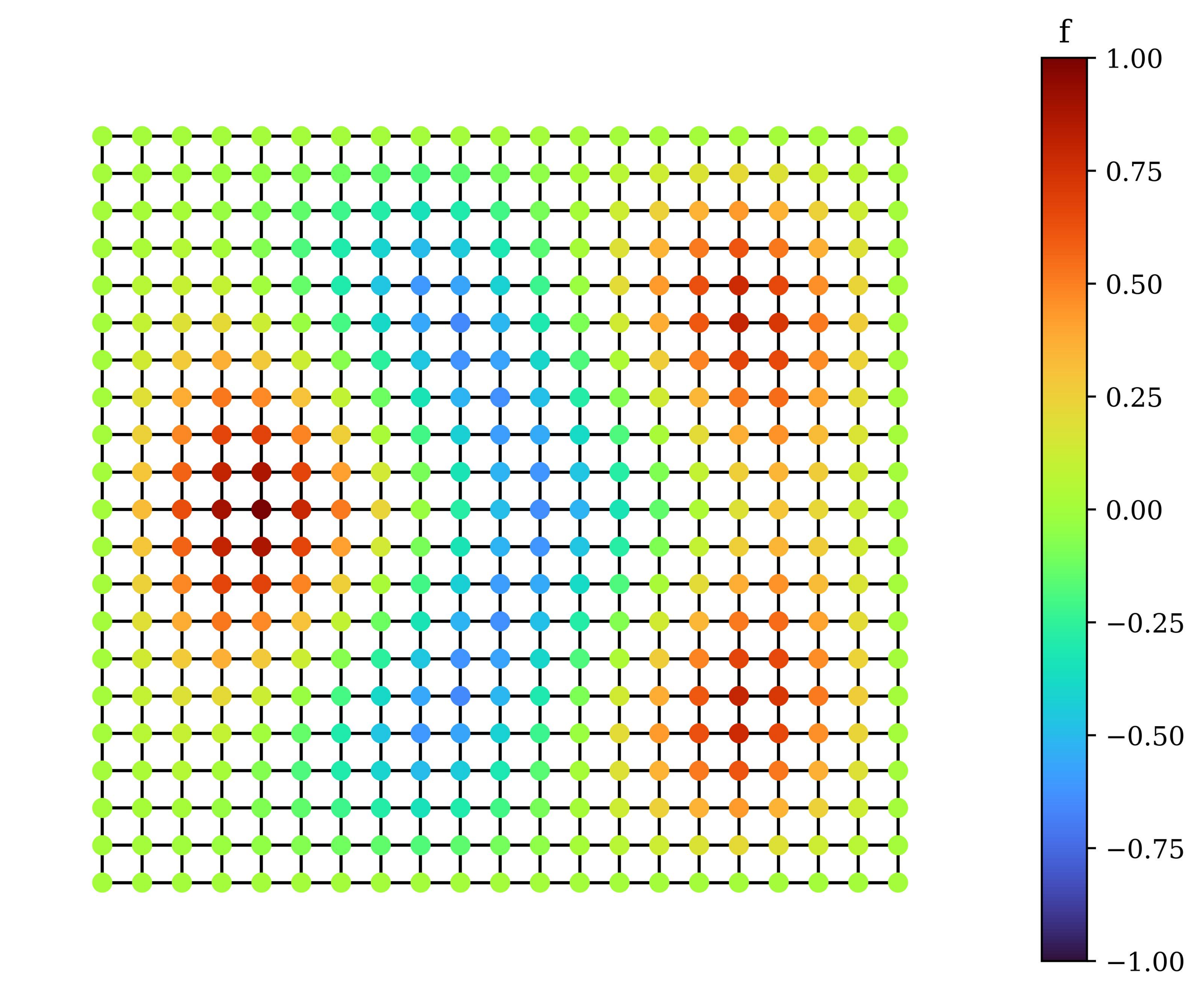}
      \includegraphics[width=\textwidth,height=2.5cm]{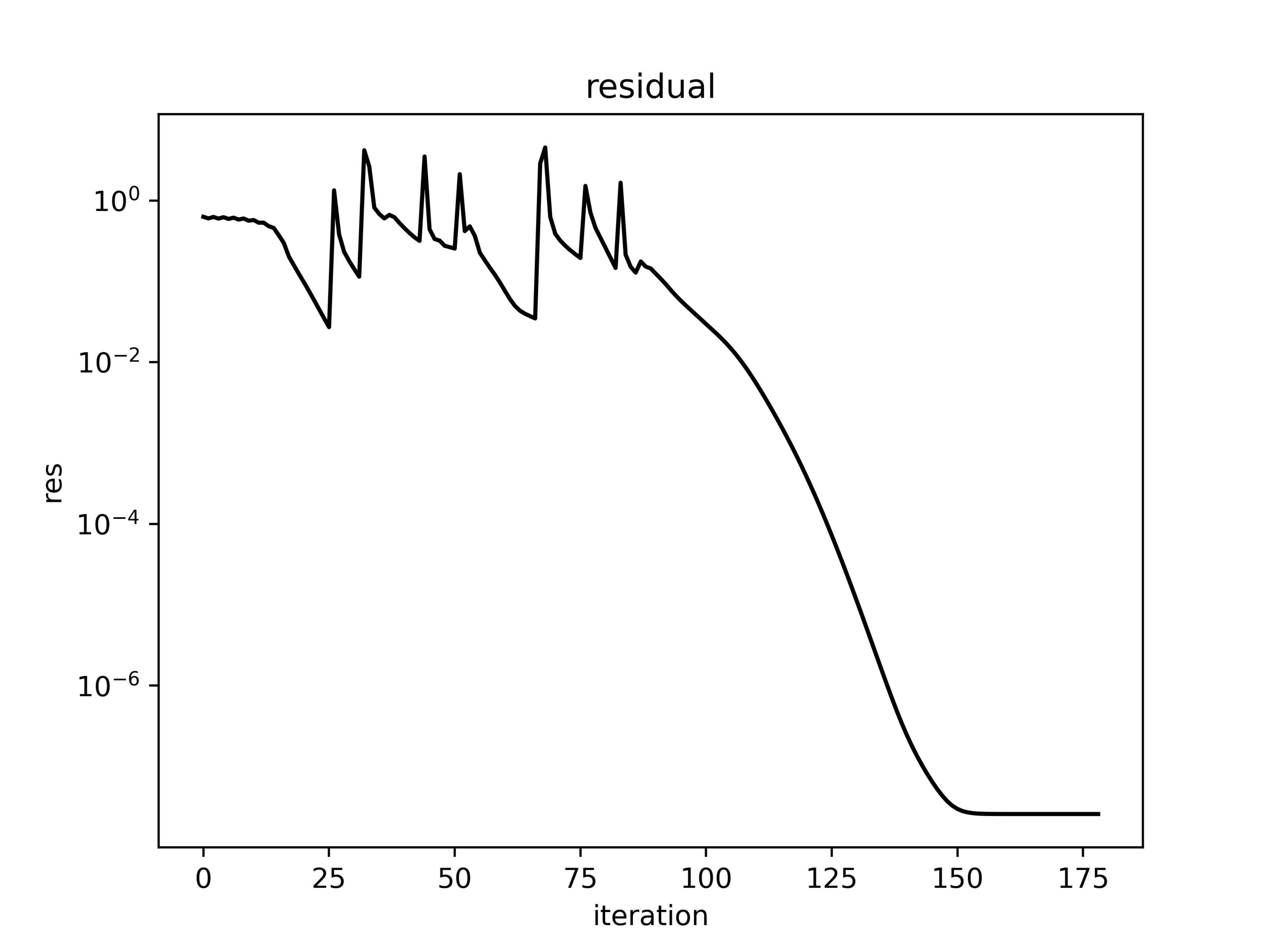}
    \end{minipage}
    \begin{minipage}{.32\textwidth}   
      \includegraphics[width=\textwidth]{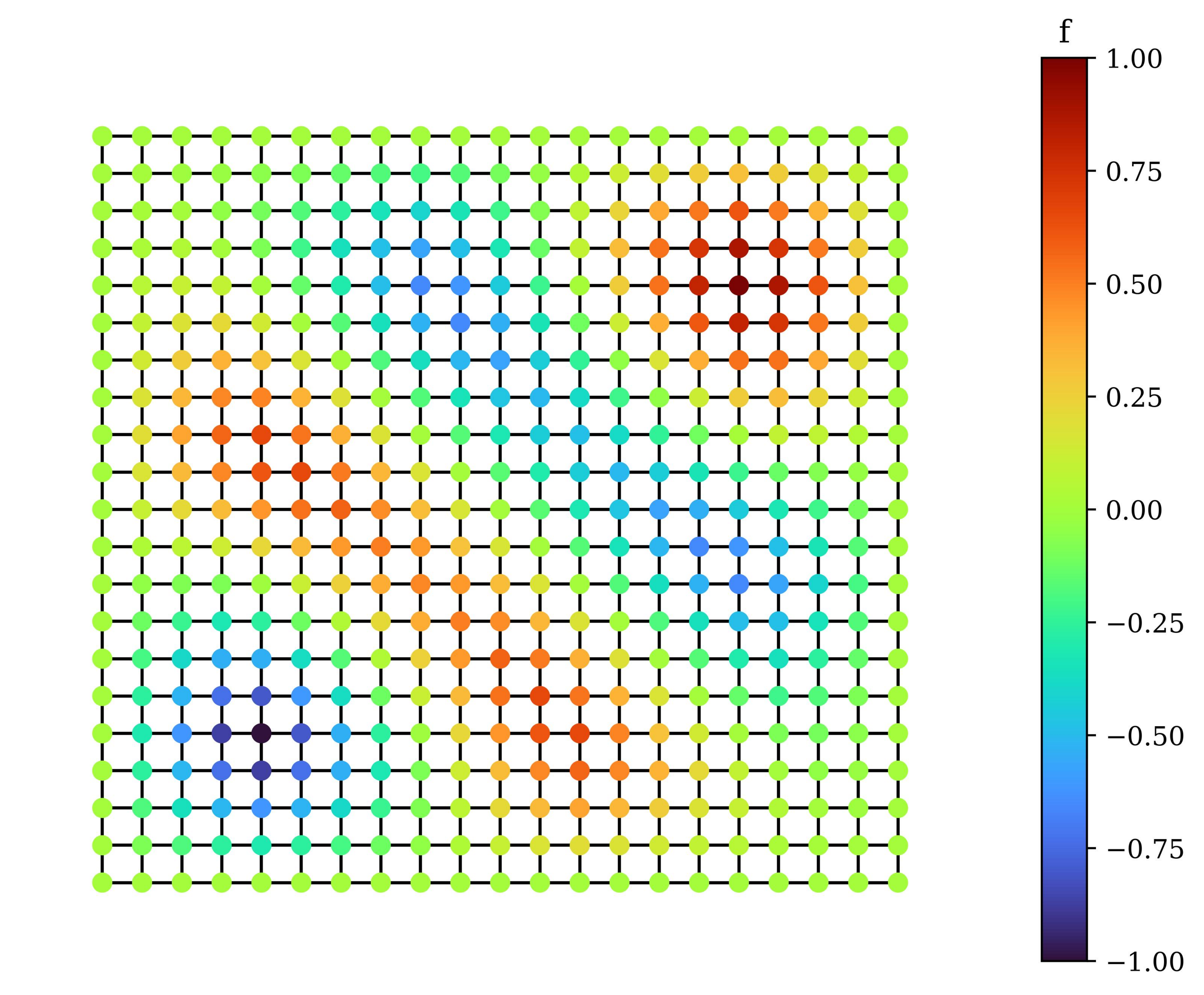}
      \includegraphics[width=\textwidth,height=2.5cm]{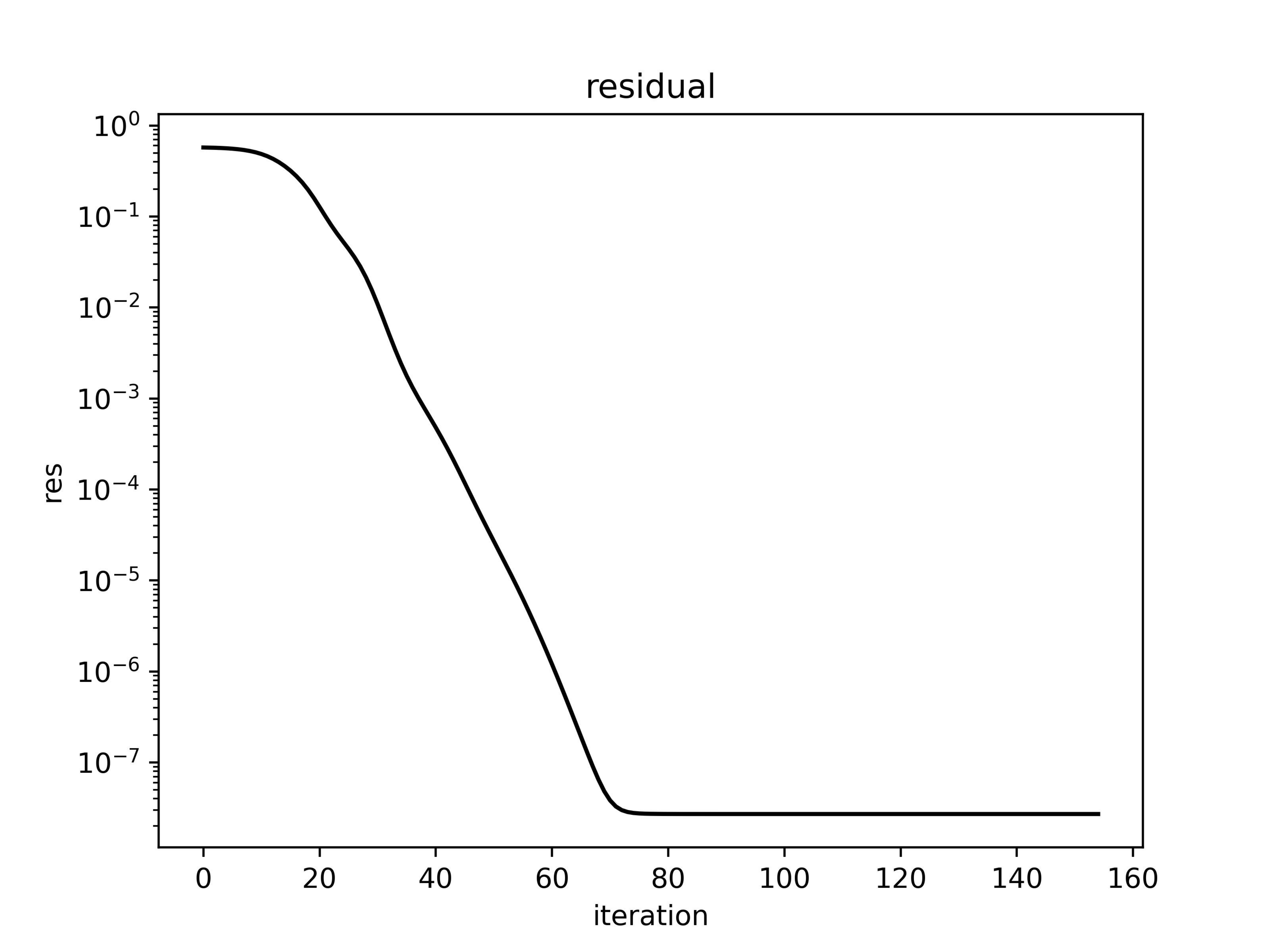}
    \end{minipage}
    \begin{minipage}{.32\textwidth}     
      \includegraphics[width=\textwidth]{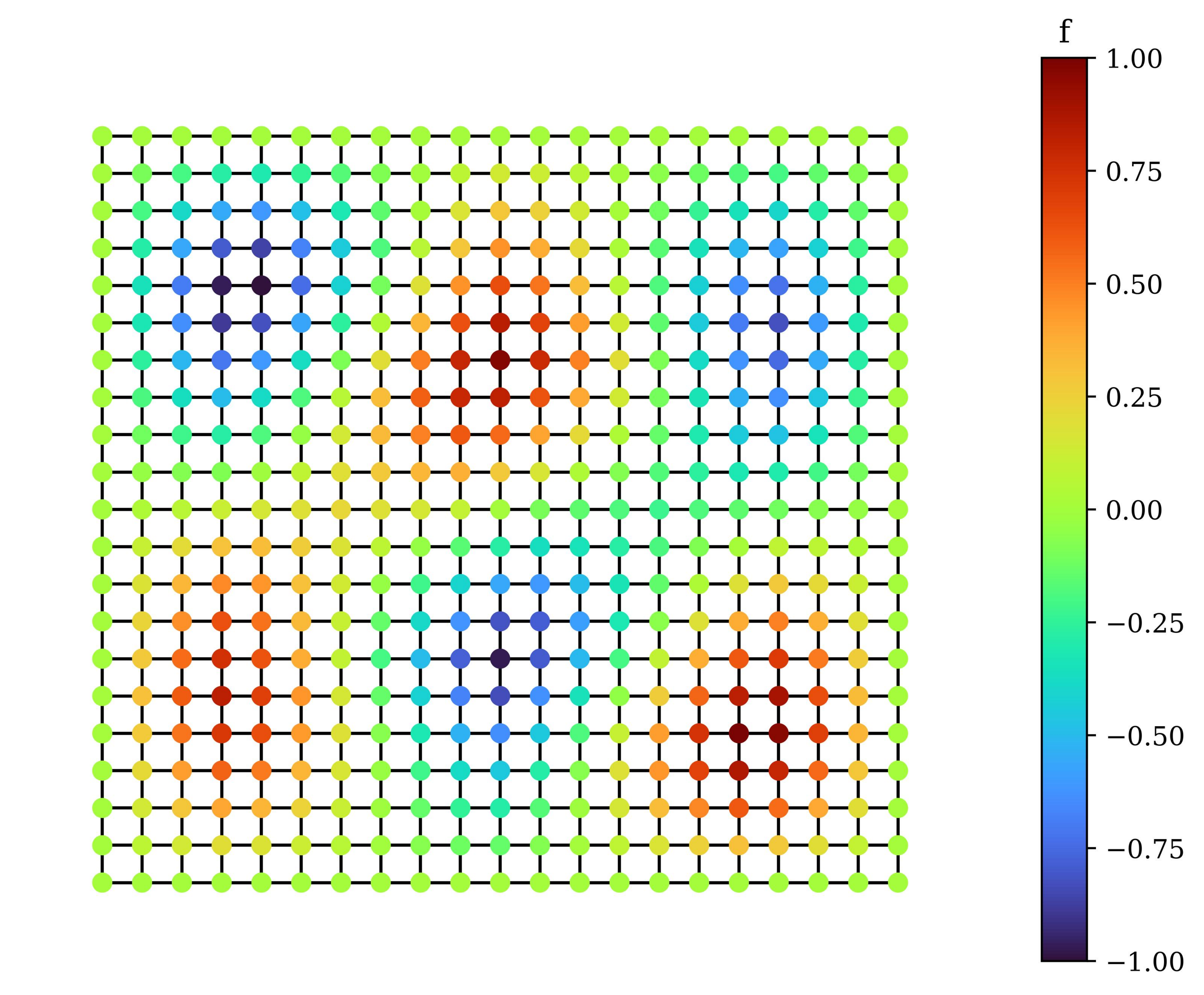}
      \includegraphics[width=\textwidth,height=2.5cm]{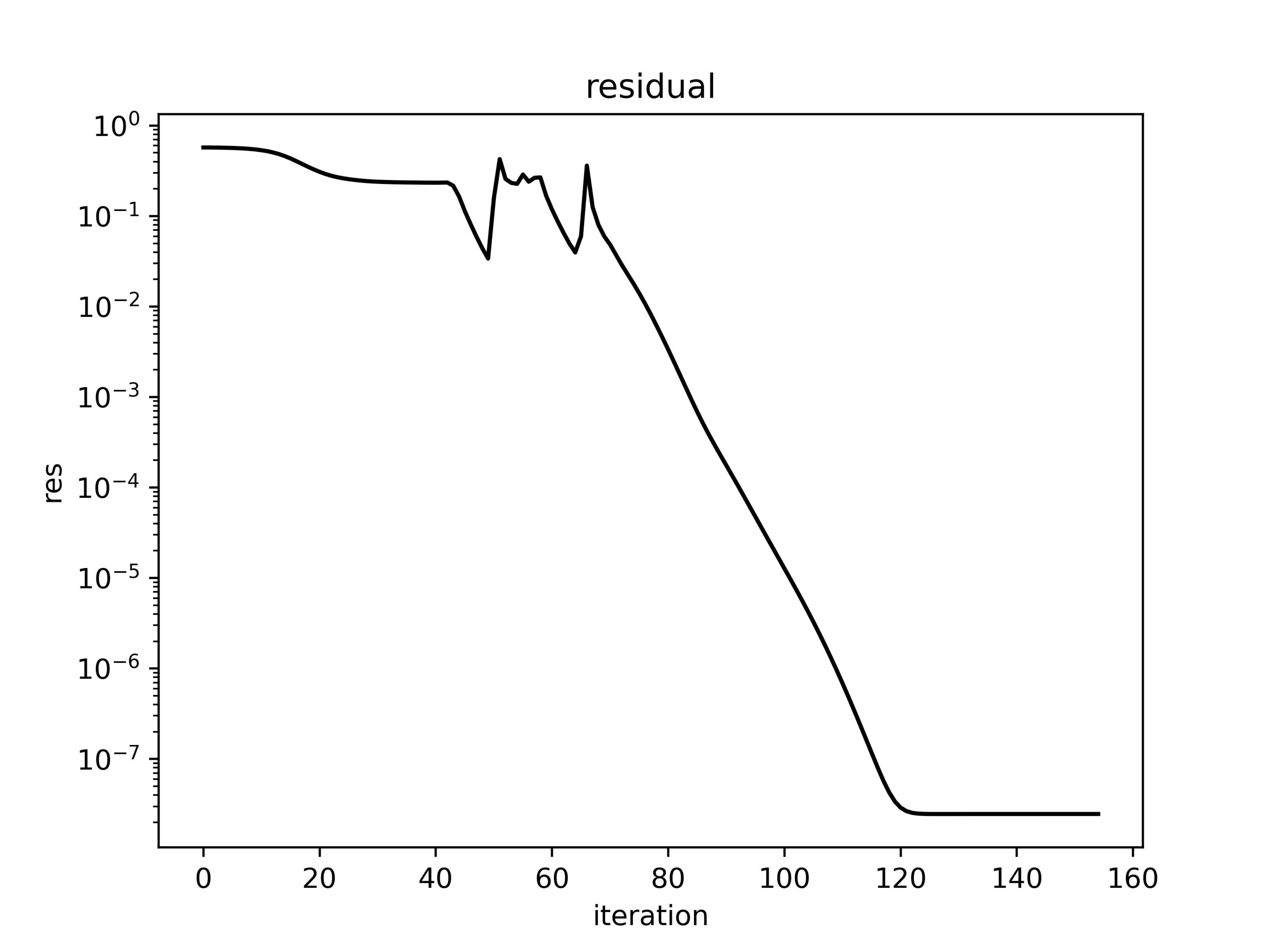}
    \end{minipage}
    \begin{minipage}{.32\textwidth}
      \includegraphics[width=\textwidth]{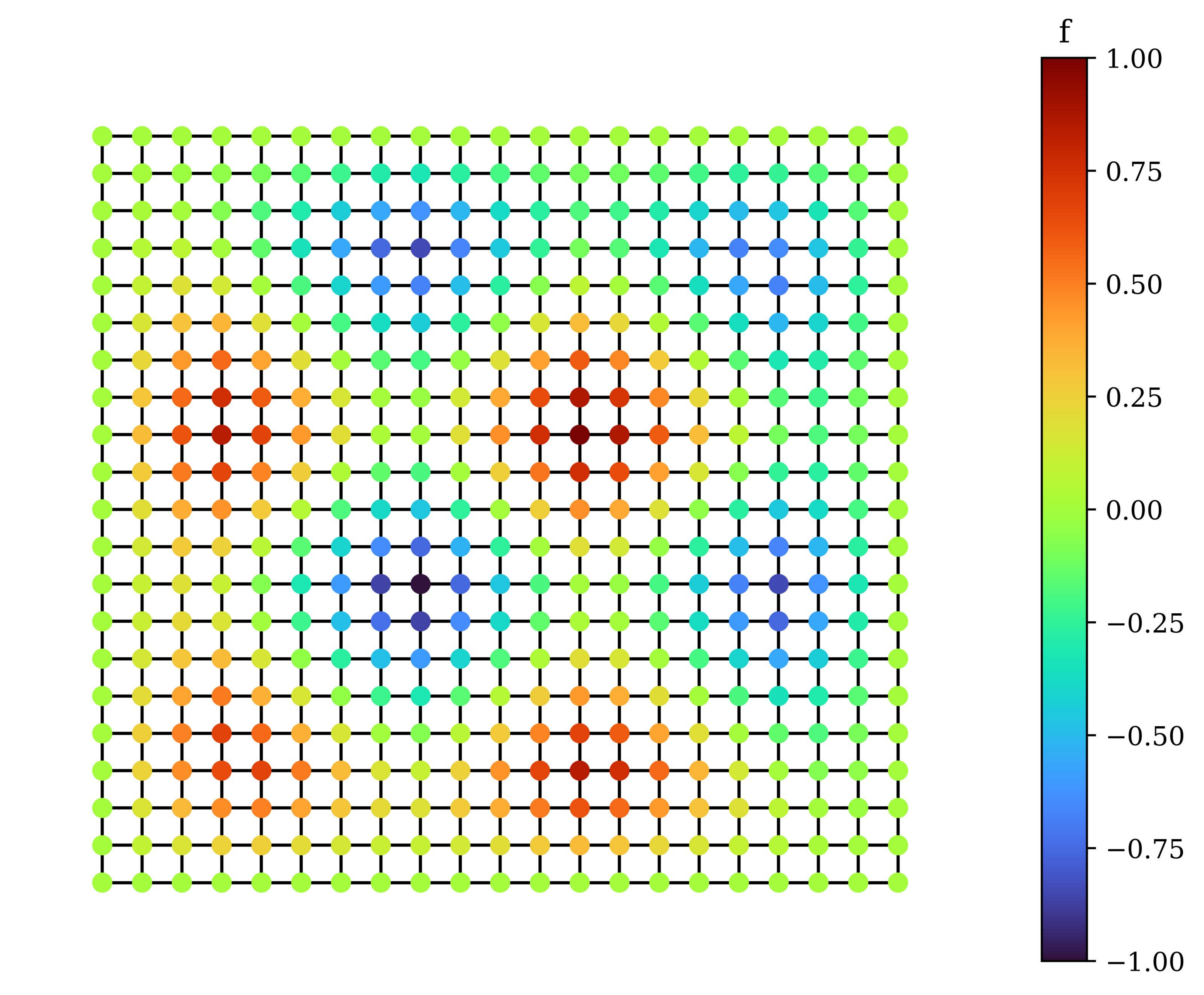}
      \includegraphics[width=\textwidth,height=2.5cm]{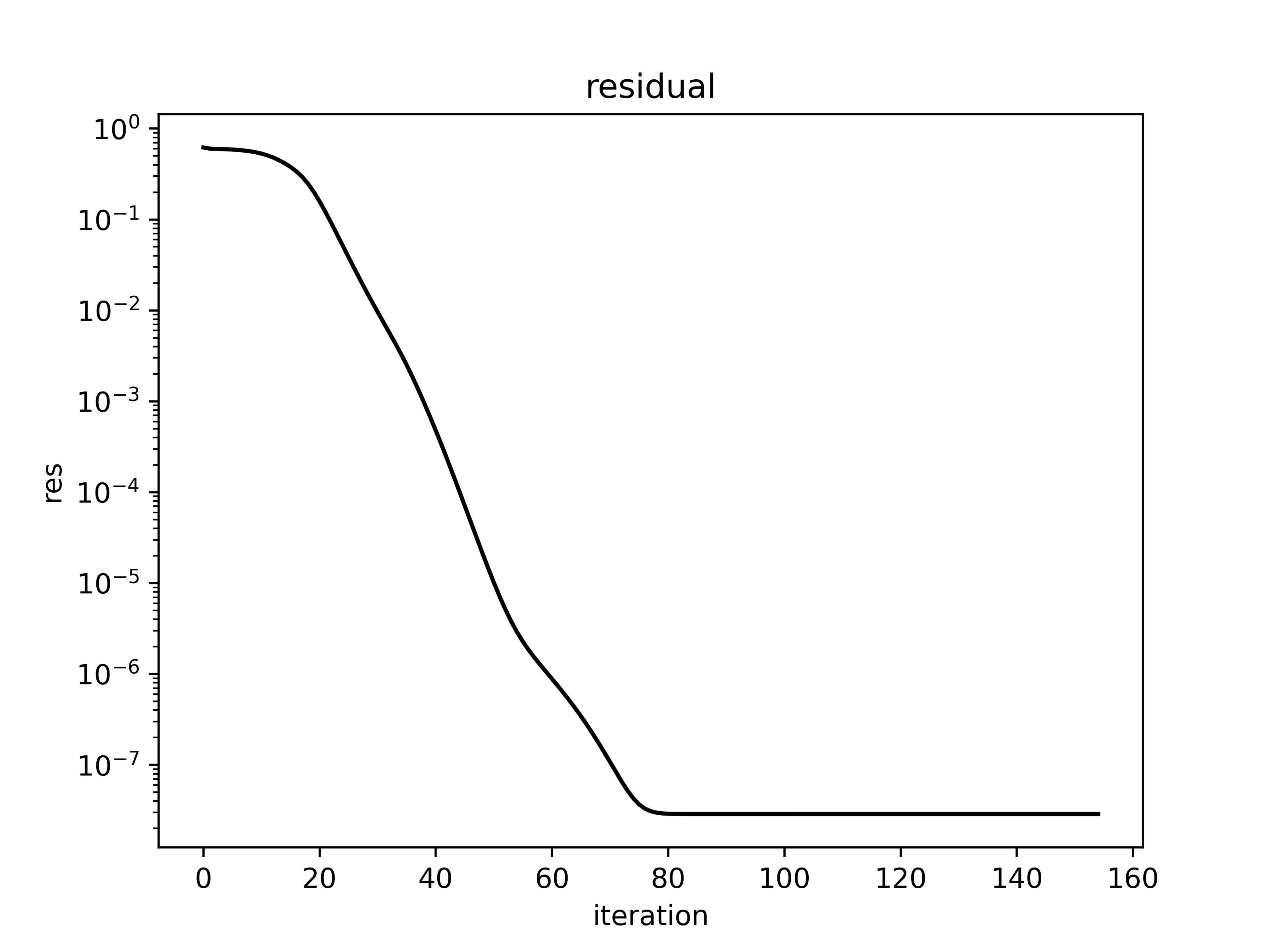}
    \end{minipage}
  \end{center}
%
\caption{First nine eigenfunctions as calculated by the proposed  method for $p=3$ displayed in sequential order from $k=1$ to $k=9$ (left to right, top to bottom). For each $k$, the top panel shows the nodal values of the eigenfunctions, while the bottom panel reports the behavior of the log residual defined in eq.~\eqref{residual} as a function of time steps (iterations) $n$.}\label{Fig-Numerical-test}
\end{figure}

Figure~\ref{Fig-Numerical-test}  shows our numerical results. Looking at the behavior of the residual, we note that, in most cases ($k=1$, $k=2$, $k=3$, $k=4$, $k=7$, $k=9$), convergence towards equilibrium is smooth and fast. However, for $k=5$, $k=6$, $k=8$, significant oscillations appear in the pre-asymptotic phase but disappear quickly and convergence of the discrete gradient flow proceeds smoothly after that.
We must recall here that for $k=1$ Theorem~\ref{Thm_saddle_point_1st_pp_eigenpair} ensures that the energy function $\Eps_{p,1}$ has only one saddle point and the proposed algorithm is expected to converge. However, for $k>1$ nothing is known. In particular, if the eigenvalues are not simple, the energy function may lose differentiability and the ODE trajectories identified by the gradient flow can intersect, potentially leading to an oscillatory behavior of the discrete method. 
In other cases we observe experimentally a pre-asymptotic oscillatory behaviour followed by monotonic convergence towards stationarity.
This behavior can be justified empirically postulating that the time step becomes large enough to jump over discontinuities and, by chance, the numerical scheme picks an appropriate trajectory thus carrying the calculations to convergence. On the other hand, it is well known that also when the gradient flow is smooth, the discrete gradient descent can stagnate.  We should note here that, unlike in the linear ($p=2$) case, we have no means at the moment to identify the position in the spectrum towards which we converge, and this adds to the difficulty of the numerical calculation of the $p$-eigenpairs.
\section{Conclusions}

We would like to conclude with some notes and a short discussion on a number of open problems that are worth addressing in future research.
We would first point out that we have not provided any theoretical study of the convergence of the continuous gradient flows, nor of the numerical schemes, that we have proposed. Thus, a detailed investigation of this theme is needed to definitely validate our approach.  
The second observation we would like to mention is related to differentiable saddle points. Any $p$-Laplacian eigenpair $(\eigenval,\eigenfunction)$ corresponding to a smooth saddle point of the $k$-th energy function can be fully characterized in a neighborhood of $\eigenfunction$ in terms of the behavior of the $p$-Rayleigh quotient. Indeed, the fact that $(\optedgeweight,\optnodeweight)$ is a differentiable saddle point implies that the eigenvalue $\eigenval[{[\optedgeweight,\optnodeweight],k}]$ is simple, yielding $\morse_f(\rayl_p)=k-1$, $\morse_f(-\rayl_p)=N-k$ by \Cref{increasing_directions}.
As a consequence, differently from the local min-max algorithm presented in \cite{Yao2007NumericalMF}, with our approach we can compute directly a $p$-Laplacian eigenpair $(\eigenval,\eigenfunction)$ such that $\morse_f(\rayl_p)=k-1$, $\morse_f(-\rayl_p)=N-k$ without the need of computing a whole sequence of $p$-Laplacian eigenpairs having linear index in $\{1,\dots,k-1\}$\,.
The second point we would like to stress is that the definition of the energy functions $\Eps_{p,k}(\edgeweight,\nodeweight)$ can be easily extended to the case $p=\infty$ by setting $p/p-2=1$ in the expression of $M_{\nodeset,p}$ and $M_{\edgeset,p}$. We will provide a more detailed study of this case in a future work.
Finally, we would like to recall that all of our results hold only in the case $p>2$, thus leaving the open problem of the case $1<p<2$. 
However we would like to recall a recent duality result, presented in~\cite{zhang2021discrete, tudisco2022nonlinear}, that relates $p$-eigenpairs on the nodes to $q$-eigenpairs on the edges ($p,q$ conjugate). This result indeed could lead to an extension of our approach to the case $p\in[1,2)$, up to a more detailed investigation. 
In particular, considering the eigenvalue problem given by the critical point equation of the $q$-Rayleigh quotient $\rayl_q^{\edgeset}$ defined on the set of edge functions $\mathcal{H}(\tilde{\edgeset})=\{G:\tilde{\edgeset}\rightarrow\R\}$ as $\rayl_q^{\edgeset}(G):=\|\incidence^T G\|_q^q/\|G\|_q^q$.
A critical pair (value, point) $(\eta,G)$ of $\rayl_q^{\edgeset}$ is regarded as a $q$-eigenpair on the edges.
Note that $(\eta,G)$ is a $q$-eigenpair if it satisfies the nonlinear eigenvalue equation:
\begin{equation}\label{eig_eq_on_edges}
  \left({\incidence}
    \left|{\incidence}^TG\right|^{q-2}
    {\incidence}^TG\right)(u,v)
  =\eta |G(u,v)|^{q-2}G(u,v)\qquad \forall (u,v)\in\tilde{\edgeset}
\end{equation}
In~\cite{zhang2021discrete, tudisco2022nonlinear} the authors show by duality that the nonzero critical values and points of $\rayl_q^{\edgeset}$ correspond to the nonzero critical values and points of $\rayl_p$, where $p$ is the conjugate of $q$.
In particular, the authors prove that if $(\eigenval,\eigenfunction)$ is an eigenpair of $\plap$ with $\eigenval\neq 0$, then $(\eigenval^{\frac{q}{p}},|\incidence f|^{p-2}\incidence f)$ is a $q$-eigenpair on the edges. Viceversa, if $(\eta,G)$ is a $q$-eigenpair on the edges with $\eta\neq 0$, then $\big(\eta^{\frac{p}{q}},|{\incidence}^TG|^{q-2}{\incidence}^TG\big)$ is a $\plap$-eigenpair.
Using these facts, we observe that equation \eqref{eig_eq_on_edges} can be reformulated in terms of a generalized linear eigenvalue problem defined on the function space $\mathcal{H}(\tilde{\edgeset})$ that can be investigated by means of spectral energy functions analogously to what we did for the primal problem when $p>2$.
In particular, note that the conjecture about the validity of our strategy in the $q=\infty$ case corresponds to the extremal case $p=1$\,.


\bibliographystyle{siamplain}
\bibliography{strings,references.bib}

\appendix
\section{Technical results}\label{sec:appendix}

\begin{proof}[Proof of \Cref{Lemma_Morse_linear_case}]
Let us complete $f$ to a $\nodeweight$ and $\lap[\edgeweight]$-orthogonal basis by taking a basis of eigenfucntions, i.e. take $\{f_i\}_{i=1}^N$ as follows: $\{f_i\}_{i=1}^{k-1}$ are eigenvectors relative to the first $k-1$ well defined eigenvalues, $f=f_k$ and $\{f_i\}_{i=k}^{k+m-1}$ are eigenvectors relative to $\eigenval_k$, including a base of the subspace $\Ker(\lap[\edgeweight])\cap \Ker(\nodeweight)$, $\{f_i\}_{i=k+m}^{N}$ are the eigenvectors relative to the well defined eigenvalues $\eigenval_i>\eigenval_k$. The eigenvectors relative to the well defined eigenvalues, except the base of $\Ker(\lap[\edgeweight])\cap \Ker(\nodeweight)$, are chosen in 
$\big(\Ker(\lap[\edgeweight])\cap \Ker(\nodeweight)\big)^{\perp}$\,.
Observe that  $T_f(S_{2,\nodeweight})=\Span\{\eigenfunction_i\}_{i\neq k}$, indeed $Tg_{\eigenfunction}(S_{2,\nodeweight})=\{\xi | \langle\nodeweight \odot\eigenfunction,\xi\rangle=0 \}$ and $\{\eigenfunction_i(\edgeweight,\nodeweight)\}_i$ is a $\nodeweight$-othogonal base of the space.

Now note that, since $\rayl_{2,\mu,\nu}$ is $0$-homogeneous, we can limit ourselves to study the behavior of the Hessian on the tangent space to the sphere $S_{2,\nodeweight}$ at the point $f$.
In this space the following implications hold:
\begin{equation}
\begin{aligned}
\xi\in\Span\{\eigenfunction_i(\edgeweight, \nodeweight)|\, i<k\}\quad  \Longrightarrow \quad  \langle \xi, H_{\rayl_{2,\mu,\nu}}(f) \xi \rangle <0\\
 \xi\in\Span\{\eigenfunction_i(\edgeweight, \nodeweight)|\,i>k+m-1\} \quad \Longrightarrow \quad  \langle \xi, H_{\rayl_{2,\mu,\nu}}(f) \xi \rangle >0\\
 \xi\in\Span\{\eigenfunction_i(\edgeweight, \nodeweight)|\,k < i \leq k+m-1\} \quad \Longrightarrow \quad  \langle \xi, H_{\rayl_{2,\mu,\nu}}(f) \xi \rangle =0\\
\end{aligned}
\end{equation}
where $H_{\rayl_{2,\mu,\nu}}(f)$ denotes the Hessian of $\rayl_{2,\mu,\nu}$ at the point $f$.
To prove the last statement, let $\xi=\sum_{i\neq k} \alpha_i\eigenfunction_i(\edgeweight,\nodeweight)$
and recall that if $i\neq j$, then $\langle\edgeweight\odot\incidence\eigenfunction_i, \incidence\eigenfunction_j\rangle=0$ and $\langle\nodeweight \odot \eigenfunction_i,\eigenfunction_j\rangle=0$. Moreover, the first derivative of the Rayleigh quotient $\rayl_{2,\mu,\nu}$ in $f$ is zero because $f$ is an eigenvector. All this information leads to the following expression of the second derivative of the Rayleight quotient (we refer to the first part of the proof of \Cref{increasing_directions} for the details about the computations)
\begin{equation*}
\begin{aligned}
   \langle \xi, H_{\rayl_{2,\mu,\nu}}(f) \xi \rangle =&      \frac{\partial^2}{\partial\epsilon^2}
    \left[
      \frac{\|\incidence(\eigenfunction+\epsilon\xi)\|_{2,\edgeweight}^2}
           {\|\eigenfunction+\epsilon\xi\|_{2,\nodeweight}^2}
    \right]_{\epsilon=0}\\
    =&     \frac{2}{\|\eigenfunction\|_{2,\nodeweight}^2}
      \left[
        \langle\edgeweight\odot\incidence\xi,\incidence\xi\rangle
        -\frac{\|\incidence\eigenfunction\|_{2,\edgeweight}^2}
              {\|\eigenfunction\|_{2,\nodeweight}^2}
        \langle\nodeweight\odot\xi,\xi\rangle
      \right],
      \end{aligned}
\end{equation*}
In particular, recalling the expression of $\xi$, we can provide the following equality that allows to conclude the proof of the lemma:
\begin{align*}
\langle \xi, H_{\rayl_{2,\mu,\nu}}(f) \xi \rangle &=
\frac{2}{\|\eigenfunction\|_{2,\nodeweight}^2}\sum_{i\neq k}\sum_{j\neq k}\alpha_i\alpha_j\bigg(\langle\edgeweight\odot\incidence\eigenfunction_i, \incidence\eigenfunction_j\rangle-\eigenval_k\langle\nodeweight \odot \eigenfunction_i,\eigenfunction_j\rangle\bigg)\\
&=
\frac{2}{\|\eigenfunction\|_{2,\nodeweight}^2}\sum_{i\neq k}\alpha_i^2  \bigg(\langle\edgeweight\odot\incidence\eigenfunction_i, \incidence\eigenfunction_i\rangle-\eigenval_k\langle\nodeweight \odot \eigenfunction_i,\eigenfunction_i\rangle\bigg)
\end{align*}
In the last equality observe that if $f_i$ is an eigenfunction corresponding to an eigenvalue $\eigenval_i$ with $f_i \not\in \Ker\big(\mathrm{diag}(\nodeweight)\big)\cap \Ker\big(\lap(\edgeweight)\big)$, then 
\begin{equation*}
\bigg(\langle\edgeweight\odot\incidence\eigenfunction_i, \incidence\eigenfunction_i\rangle-\eigenval_h\langle\nodeweight\odot \eigenfunction_i,\eigenfunction_i\rangle\bigg)=\|\eigenfunction_i\|_{2,\nodeweight}^2 \big(\eigenval_i-\eigenval_k\big)\,,
\end{equation*}
i.e., $f_i$ is an increasing or a decreasing direction of $\rayl_{2,\edgeweight,\nodeweight}$ at the point $f$ according to the inequalities $\eigenval_i>\eigenval_k$ or $\eigenval_i<\eigenval_k$.
Moreover if $\lambda_i=\lambda_k$, or $f_i\in \Ker(\lap[\edgeweight])\cap \Ker\big(\mathrm{diag}(\nodeweight)\big)$ it is trivial to observe that the second derivative of $\rayl_{2,\edgeweight,\nodeweight}$ in the direction of $f_i$ at the point $f$ is zero,  i.e.: 
\begin{equation*}
\bigg(\langle\edgeweight\odot\incidence\eigenfunction_i, \incidence\eigenfunction_i\rangle-\eigenval_h\langle\nodeweight\odot \eigenfunction_i,\eigenfunction_i\rangle\bigg)=0\,.
\end{equation*}
\end{proof}

\begin{proof}[Proof of \Cref{increasing_directions}]
    We aim at proving that $\forall \xi\in T_{\eigenfunction}(S_p)=T_{\eigenfunction}(S_{\nodeweight})$ we have:
\begin{equation*}
  \frac{\partial^{2}}{\partial \epsilon^{2}}
  \bigg(
  \frac{\|\incidence(\eigenfunction+\epsilon \xi)\|_p^p}{\|\eigenfunction+\epsilon \xi\|_p^p}
  \bigg)\bigg|_{\epsilon=0}
  =\frac{p(p-1)}{2} \frac{\partial^{2}}{\partial \epsilon^{2}}
  \bigg(
  \frac{\|\incidence(\eigenfunction+\epsilon \xi)\|_{2,\edgeweight}^2}{\|\eigenfunction+\epsilon \xi\|_{2,\nodeweight}^2}
  \bigg)\bigg|_{\epsilon=0}\,.
\end{equation*}
Because of the equivalence of the $p$-Laplacian and weighted Laplacian eigenvalue problems, $f$ is a critical point for both Rayleigh quotients $\rayl_p$ and $\rayl_{2,\edgeweight,\nodeweight}$, i.e., and hence their first derivative is zero:
\begin{equation}\label{increasing_directions_eq.1}
  \begin{aligned}
    0&=\frac{\partial}{\partial \epsilon}\bigg(\frac{\|\incidence(\eigenfunction+\epsilon \xi)\|_p^p}{\|\eigenfunction+\epsilon \xi\|_p^p}\bigg)\bigg|_{\epsilon=0}=\frac{p}{\|\eigenfunction\|_p^p}\bigg(\langle|\incidence \eigenfunction|^{p-2}\odot\incidence \eigenfunction, \incidence \xi\rangle-\frac{\|\incidence \eigenfunction\|_p^p}{\|\eigenfunction\|_p^p}\langle|\eigenfunction|^{p-2}\odot\eigenfunction,\xi\rangle\bigg)\\
    0&=\frac{\partial}{\partial \epsilon}\bigg(\frac{\|\incidence(\eigenfunction+\epsilon \xi)\|_{2,\edgeweight}^2}{\|\eigenfunction+\epsilon \xi\|_{2,\nodeweight}^2}\bigg)\bigg|_{\epsilon=0}=\frac{2}{\|\eigenfunction\|_{2,\nodeweight}^2}\bigg(\langle\edgeweight \odot \incidence \eigenfunction, \incidence\xi\rangle-\frac{\|\incidence \eigenfunction\|_{2,\edgeweight}^2}{\|\eigenfunction\|_{2,\nodeweight}^2}\langle\nodeweight \odot \eigenfunction,\xi\rangle\bigg)
  \end{aligned}
\end{equation}
We note that, since $\xi\in T_{\eigenfunction}(S_p)=T_{\eigenfunction}(S_{\nodeweight})$, we have 
\begin{equation}\label{increasing_directions_eq.2}
  \frac{\partial}{\partial \epsilon}
  \|\eigenfunction+\epsilon\xi\|_p^p\big|_{\epsilon=0}
  =\frac{\partial}{\partial\epsilon}
  \|\eigenfunction+\epsilon\xi\|_{2,\nodeweight}^2\big|_{\epsilon=0}
  =C\langle|\eigenfunction|^{p-2}\odot\eigenfunction,\xi\rangle
  =C\langle\nodeweight \odot \eigenfunction,\xi\rangle=0\,,
\end{equation}
where $C$ is an appropriate constant.
Now, for any $x,y\in\R$, we can calculate the following derivative
\begin{equation}\label{increasing_directions_eq.3}
    \frac{\partial |x+\epsilon y|^{p-2}(x+\epsilon y)}{\partial \epsilon}\big|_{\epsilon=0}=(p-2)|x|^{p-3}\frac{(x)^2}{|x|}y+|x|^{p-2}y=(p-1)|x|^{p-2}y\,.    
\end{equation}
Differentiating~\eqref{increasing_directions_eq.1}, using~\eqref{increasing_directions_eq.2}, and \eqref{increasing_directions_eq.3}, and recalling that $|\eigenfunction+\epsilon \xi|^{p-2}\odot(\eigenfunction+\epsilon \xi)$  and $|\incidence(\eigenfunction+\epsilon \xi)|^{p-2}\odot\left(\incidence (\eigenfunction+\epsilon \xi)\right)$
are entrywise products, we obtain:
\begin{equation}\label{eq_second_derivatives}
  \begin{aligned}
    \frac{\partial^2}{\partial\epsilon^2}
    \left[
      \frac{\|\incidence(\eigenfunction+\epsilon\xi)\|_p^p}
      {\|\eigenfunction+\epsilon\xi\|_p^p}
    \right]_{\epsilon=0}
    &=\frac{p(p-1)}{\|\eigenfunction\|_p^p}
      \left[
        \langle
        |\incidence\eigenfunction|^{p-2}\odot\incidence\xi,\incidence\xi
        \rangle
        -\frac{\|\incidence\eigenfunction\|_p^p}{\|\eigenfunction\|_p^p}
        \langle|\eigenfunction|^{p-2}\odot \xi,\xi\rangle
      \right]\\
    \frac{\partial^2}{\partial\epsilon^2}
    \left[
      \frac{\|\incidence(\eigenfunction+\epsilon\xi\|_{2,\edgeweight}^2}
           {\|\eigenfunction+\epsilon\xi\|_{2,\nodeweight}^2}
    \right]_{\epsilon=0}
    &=\frac{2}{\|\eigenfunction\|_{2,\nodeweight}^2}
      \left[
        \langle\edgeweight\odot\incidence\xi,\incidence\xi\rangle
        -\frac{\|\incidence\eigenfunction\|_{2,\edgeweight}^2}
              {\|\eigenfunction\|_{2,\nodeweight}^2}
        \langle\nodeweight\odot\xi,\xi\rangle
      \right]\\
    &=\frac{2}{\|\eigenfunction\|_{p}^p}
      \left[
        \langle
          |\incidence\eigenfunction|^{p-2}\odot\incidence\xi,\incidence\xi
        \rangle
        -\frac{\|\incidence\eigenfunction\|_{p}^p}{\|\eigenfunction\|_{p}^p}
        \langle|\eigenfunction|^{p-2}\odot\xi,\xi\rangle
      \right]
  \end{aligned}
\end{equation}
which yields the desired equality. The Thesis of the proposition follows from \Cref{Lemma_Morse_linear_case}\,.
\end{proof}

\begin{proof}[Proof of \Cref{lemma_eigenvalues_derivative}]
Observe first of all that if $\lambda_k(\edgeweight)$ is differentiable in $\edgeweight^*$ then $\lambda_k(\edgeweight^*)$ is simple and thus $f_k^*$ is uniquely defined, see \cite{kato2013perturbation}. By the chain rule it is enough to show that
\begin{equation}
    \partial_{\edgeweight_{uv}}\lambda_k(\mu^*)=\frac{\partial\edgeweight_{uv}\Big({\eigenfunction_k^*}^T \grad^T\mathrm{diag}(\edgeweight^*)\grad{\eigenfunction_k^*}\Big)}{2\|\eigenfunction_k^*\|_{2,\nodeweight}^2}=\frac{|\grad f_k^*(u,v)|^2}{2\|f_k^*\|^2_{2,\nodeweight}}\,.
    \end{equation}
To prove the last equality, we differentiate both the terms of the eigenvalue equation with respect to $\edgeweight[uv]$:
\begin{equation}
    \begin{aligned}
        \partial_{ \edgeweight[uv]}\Big(\lap[\edgeweight^*]\eigenfunction_k^*\Big)&=\partial_{ \edgeweight[uv]}\Big(\lambda_k^* \mathrm{diag}(\nodeweight^*)\eigenfunction_k^*\Big)\\
\partial_{\edgeweight[uv]}\big(\lap[\edgeweight^*]\big)\eigenfunction_k^*+\lap[\edgeweight^*]\partial_{\edgeweight[uv]}\big(\eigenfunction_k^*\big)&=\partial_{\edgeweight[uv]}\big(\lambda_k^*\big) \mathrm{diag}(\nodeweight^*)\eigenfunction_k^*+ \lambda_k^*\mathrm{diag}(\nodeweight^*)\partial_{\edgeweight[uv]}\big(\eigenfunction_k^*\big)\,.
    \end{aligned}
\end{equation}
Then multiply both terms by $\eigenfunction_k^*$ and remember $\lap[\edgeweight^*]\eigenfunction_k^*=\lap[\edgeweight^*]^T\eigenfunction_k^*=\lambda_k^* \mathrm{diag}(\nodeweight^*) \eigenfunction_k^*$ and $\lap[\edgeweight^*]=\grad^T\mathrm{diag}(\edgeweight^*)\grad$ 
\begin{equation}
    \begin{aligned}
{\eigenfunction_k^*}^T \partial_{ \edgeweight[uv]}\big(\lap[\edgeweight^*]\big)\eigenfunction_k^*+&\lambda_k^* {\eigenfunction_k^*}^T \mathrm{diag}(\nodeweight^*) \partial_{\edgeweight[uv]}\big(\eigenfunction_k^*\big)=\\
&=\partial_{\edgeweight[uv]}\big(\lambda_k^*\big) {\eigenfunction_k^*}^T\mathrm{diag}(\nodeweight^*)\eigenfunction_k^*+ \lambda_k^*{\eigenfunction_k^*}^T\mathrm{diag}(\nodeweight^*) \partial_{\edgeweight[uv]}\big(\eigenfunction_k^*\big)
  \end{aligned}
\end{equation}
i.e.
\begin{equation}%
{\eigenfunction_k^*}^T \grad^T \mathrm{diag}(e_{uv}) \grad \eigenfunction_k^*=\partial_{\edgeweight[uv]}\big(\lambda_k^*\big) {\eigenfunction_k^*}^T\mathrm{diag}(\nodeweight^*)\eigenfunction_k^*
\end{equation}
where $e_{uv}$ is the characteristic function of the edge $(u,v)$, the proof is concluded by means of the chain rule. The derivative in $\nodeweight$ can be obtained analogously.
\end{proof}

\begin{proof}[Proof of \Cref{p2_first_eigen_charact}]
In this proof we use the maximum principle result from \cite{Park2011} reported in the following Lemma.
We point out that our definition of the $p$-Laplacian operator (see Def.~\ref{Def_p-lap}) matches the definition of the generalized $p$-Laplacian operator used in the maximum principle in~\cite{Park2011}. 
\begin{lemma}[from~\cite{Park2011}]\label{maximum_principle}
  If $f,g:\nodeset\rightarrow\R$ satisfy $\plap f(u)> \plap g(u)$,
then $f(u)\geq g(u)$ for any $u\in\internalnodes$.
\end{lemma}
We first observe that, for $\nodeweight\neq 0$, the Rayleigh quotient $\rayl_{p,2,\nodeweight}$ is always well defined if we admit that it takes values in $[0,+\infty]$. Indeed, if $\boundary\neq\emptyset$ then $\Ker(\incidence)=\emptyset$. If $\boundary = \emptyset$, then $\Ker(\incidence)=\mathrm{span}(\underline{1})$, where $\underline{1}$ denotes the constant vector but for any $\nodeweight\neq 0$, $\underline{1}\not\in\Ker\big(\diag(\nodeweight)\big)$.
  In any case, for all $f\neq0$ we have that  $\min\rayl_{p,2,\nodeweight}<\infty$. 
  
  Let $\eigenfunction_1$ be a minimum point of $\rayl_{p,2,\nodeweight}$ such that $\|\eigenfunction_1\|_{2,\nodeweight}=1$.
  An easy calculation shows that
  \begin{equation*}
    \rayl_{p,2,\nodeweight}(|\eigenfunction_1|)\leq\rayl_{p,2,\nodeweight}(\eigenfunction_1)\,,
  \end{equation*}
  with equality if and only if $\eigenfunction_1=\pm|\eigenfunction_1|$. Thus we can assume that $\eigenfunction_1(u)\geq0$ for all $u\in\internalnodes$. 
  If $\eigenfunction_1(u)=0$ for some $u\in\internalnodes$, then eq.~\eqref{eq:p2-L-dirichlet} and the explicit expression of $\plap$ in eq.~\eqref{explicit_plap_eig_eq} ensure that $\eigenfunction_1(v)=0$ for any $v\sim u$. As a consequence of the connectedness of the graph, this implies $\eigenfunction_1=0$ for all $u\in\internalnodes$, from which  $\|\eigenfunction_1\|_{2,\nodeweight}=0$, contradicting the initial hypothesis. 
  Now we can prove the second part of the theorem. We start from the last statement. Assume that there exists a positive eigenfunction $\eigenfunction_2>0$ such that $\rayl_{p,2,\nodeweight}(\eigenfunction_2)=\eigenval_2>\eigenval_1=\rayl_{p,2,\nodeweight}(\eigenfunction_1)$. Then there exist $t>0$ and $u_0\in\internalnodes$ such that
  \begin{equation*}
    \eigenval_2\eigenfunction_2(u)
    >t \eigenval_1\eigenfunction_1(u)\; \forall u\in\internalnodes
    \quad  \text{and} \quad
    t\eigenfunction_1(u_0)>\eigenfunction_2(u_0)\;.
  \end{equation*}
  Applying Theorem \ref{maximum_principle} to the functions $t\eigenfunction_1$ and $\eigenfunction_2$ we get a contradiction, proving that only positive eigenfunctions are associated to the first eigenvalue.
  We are left to prove that $\eigenval_1$ is simple, which implies the uniqueness of the corresponding eigenfunction $\eigenfunction_1$. Assume that there exist two positive eigenfunctions $\eigenfunction_1$ and $\eigenfunction_2$ relative to $\eigenval_1$ with $\|\eigenfunction_1\|_{2,\nodeweight}=\|\eigenfunction_2\|_{2,\nodeweight}=1$. 
  Then, the function
  \begin{equation*}
    g(u)=\left(
    \eigenfunction_1^2(u)+\eigenfunction_2^2(u)
    \right)^{\frac{1}{2}}\,,
  \end{equation*}
  has $2$-norm given by $\|g\|_{2,\nodeweight}^{p}=2^{\frac{p}{2}}$ and its gradient satisfies:
  \begin{equation*}
    \|\incidence g\|_p^p
    \leq 2^{\frac{p-2}{2}}
    \left(
      \|\incidence\eigenfunction_1\|_p^p +\|\incidence\eigenfunction_2\|_p^p
    \right)
  \end{equation*}
  with equality holding if and only if $\incidence\eigenfunction_1(u,v)=\incidence\eigenfunction_2(u,v)$ for all $(u,v)\in\edgeset$.
  To prove the last inequality, consider an edge $(u,v)$ and use first the Cauchy Schwarz inequality applied to the two vectors $\big(\eigenfunction_1(u),\eigenfunction_2(u)\big)$ $\big(\eigenfunction_1(v),\eigenfunction_2(v)\big)$ and then Jensen inequality applied to the function   $x\mapsto|x|^{\frac{p}{2}}$:
  \begin{equation*}
    \begin{aligned}
      |\incidence g(v,u)|^p
      &=\edgelength_{uv}^{p}
        \Big|
          \big(
            \eigenfunction_1(u)^2+\eigenfunction_2(u)^2
          \big)^{\frac{1}{2}}
          -
          \big(
            \eigenfunction_1(v)^2+\eigenfunction_2(v)^2
          \big)^{\frac{1}{2}}
          \Big|^p\\
      &\leq \edgelength_{uv}^{p}
        \Big|
          \big(
            \eigenfunction_1(u)-\eigenfunction_1(v)
          \big)^2
          +
          \big(
            \eigenfunction_2(u)-\eigenfunction_2(v)
          \big)^2
          \Big|^{\frac{p}{2}}\\
      &\leq\edgelength_{uv}^{p} 2^{\frac{p-2}{2}}
        \Big(
          \big|
            \eigenfunction_1(u)-\eigenfunction_1(v)
          \big|^p
          +
          \big|
            \eigenfunction_2(u)-\eigenfunction_2(v)
          \big|^p
        \Big)\\
      &= 2^{\frac{p-2}{2}}
        \big(
          |\incidence\eigenfunction_1(v,u)|^p
          +|\incidence\eigenfunction_2(v,u)|^p
        \big)
    \end{aligned}
  \end{equation*}
  where, by convexity of the function $|x|^{\frac{p}{2}}$, we have equality if and only if $\eigenfunction_1(u)-\eigenfunction_1(v)=\eigenfunction_2(u)-\eigenfunction_2(v)$.
  This means that
  \begin{equation*}
    \eigenval_1 2^{\frac{p}{2}}
    =\eigenval_1\|g\|^{p}_{2,\nodeweight}
    \leq \|\incidence g\|_p^p
    \leq 2^{\frac{p-2}{2}}
    \big(
      \|\incidence\eigenfunction_1\|_p^p
      +\|\incidence\eigenfunction_2\|_p^p
    \big)
    =\eigenval_1 2^{\frac{p}{2}}\,,
  \end{equation*}
  implying that in any edge $\eigenfunction_1(u)-\eigenfunction_1(v)=\eigenfunction_2(u)-\eigenfunction_2(v)$ and thus, by the connectedness of the graph and the assumptions on $f_1$ and $f_2$, we obtain $\eigenfunction_1=\eigenfunction_2$. 
\end{proof}

\begin{proof}[Proof of \Cref{lemma_1st_p2_eigen_gradient}]
Recall the definition of the $[p,2,\nodeweight]$-Rayleigh quotient: 
\begin{equation*}
  \rayl_{p,2,\nodeweight}(\eigenfunction):=
  \frac{\|\incidence\eigenfunction\|_p^p}
  {\|\eigenfunction\|^{p}_{2,\nodeweight}}
  =\frac{\dDsum_{(u,v)\in\tilde{\edgeset}}|\incidence\eigenfunction(u,v)|^p}
  {\Big(\dDsum_{u\in\internalnodes}
      \nodeweight(u)|\eigenfunction(u)|^2\Big)^{\frac{p}{2}}}
\end{equation*}
Recall that, given $\nodeweight\in \mathcal{M}^+(\internalnodes)\setminus\{0\}$, the first eigenvalue is characterized by 
\begin{equation*}
  \eigenval[1](\nodeweight)
  :=\min_{f}\rayl_{p,2,\nodeweight}(\eigenfunction)
  =\rayl_{p,2,\nodeweight}(\eigenfunction[\nodeweight,1],\nodeweight)\,.
\end{equation*}
The function that associates to a density $\nodeweight$ the corresponding first eigenfunction, $\eigenfunction[\nodeweight]:=\eigenfunction[{[p,2,\nodeweight],1}]$, of the $[p,2]$-Laplacian weighted in $\nodeweight$, with $\|\eigenfunction[\nodeweight]\|_{2,\nodeweight}=1$ is well defined by Theorem \ref{p2_first_eigen_charact} and  continuous by the continuity of minimizers.
Now consider the variation of $\eigenval_1$ near a point $\nodeweight_0\in \mathcal{M}^+(\internalnodes)\setminus\{0\}$. We have the following inequality:
\begin{equation*}
  \begin{aligned}
    \eigenval[1](\nodeweight_0)-\eigenval[1](\nodeweight)
    &=
    \rayl_{p,2,\nodeweight[0]}(\eigenfunction[{\nodeweight[0]}])
    -\rayl_{p,2,\nodeweight}(\eigenfunction[\nodeweight])\\
    &\leq
    \rayl_{p,2,\nodeweight[0]}(\eigenfunction[\nodeweight])
    -\rayl_{p,2,\nodeweight}(\eigenfunction[\nodeweight])
    =\partial_{\nodeweight}
    \rayl_{p,2,\nodeweight[0]}(\eigenfunction[\nodeweight])
   (\nodeweight[0]-\nodeweight)+ o(\|\nodeweight[0]-\nodeweight\|)\,,
  \end{aligned}
\end{equation*}
which implies
\begin{equation*}
  \begin{aligned}
    \limsup_{\nodeweight\rightarrow\nodeweight[0]}
    &
    \left(\eigenval[1](\nodeweight[0])
      -\eigenval[1](\nodeweight)
      -\partial_{\nodeweight}
      \rayl_{p,2,\nodeweight[0]}(\eigenfunction[{\nodeweight[0]}])
      (\nodeweight[0]-\nodeweight)
    \right)\\
    &
    \leq \limsup_{\nodeweight\rightarrow\nodeweight_0}
    \left(\partial_{\nodeweight}
      \rayl_{p,2,\nodeweight[0]}(\eigenfunction[{\nodeweight}])
      -\partial_{\nodeweight}
      \rayl_{p,2,\nodeweight[0]}(\eigenfunction[{\nodeweight[0]}])
    \right)(\nodeweight[0]-\nodeweight)=0 \,.
  \end{aligned}
\end{equation*}
Similarly we can write:
\begin{equation*}
  \begin{aligned}
    \eigenval[1](\nodeweight[0])-\eigenval[1](\nodeweight)
    &
    =
    \rayl_{p,2,\nodeweight[0]}(\eigenfunction[{\nodeweight[0]}])
    -\rayl_{p,2,\nodeweight}(\eigenfunction[\nodeweight])\\
    &
    \geq
    \rayl_{p,2,\nodeweight[0]}(\eigenfunction[{\nodeweight[0]}])
    -\rayl_{p,2,\nodeweight}(\eigenfunction[{\nodeweight[0]}])
    =\partial_{\nodeweight}
    \rayl_{p,2,\nodeweight[0]}(\eigenfunction[{\nodeweight[0]}])
    (\nodeweight[0]-\nodeweight)
    + o\left(\|\nodeweight[0]-\nodeweight \|\right)
  \end{aligned}
\end{equation*}
which implies:
\begin{equation*}
  \liminf_{\nodeweight\rightarrow\nodeweight_0}
  \left(
  \eigenval[1](\nodeweight[0])-\eigenval[1](\nodeweight)
  -\partial_{\nodeweight}
  \rayl_{p,2,\nodeweight[0]}(\eigenfunction[{\nodeweight[0]}])
  (\nodeweight[0]-\nodeweight)\right)\geq 0\,.
\end{equation*}
We can now conclude by observing:
\begin{equation*}
  \partial_{\nodeweight} \eigenval[1](\nodeweight[0])
  =\partial_{\nodeweight}
  \rayl_{p,2,\nodeweight[0]}(\eigenfunction[{\nodeweight[0]}])
  =
  -\frac{p}{2}
  \frac{\eigenval[1](\nodeweight[0])|\eigenfunction[{\nodeweight[0]}]|^2}
  {\|\eigenfunction[{\nodeweight[0]}]\|_{2,\nodeweight[0]}^2}\,.
\end{equation*}
\end{proof}


\end{document}

%% file: packages.tex
\usepackage[english]{babel}

\usepackage{color}
\usepackage{amsmath, amssymb, mathdots, accents}
\usepackage{ifthen}
\usepackage{caption,subcaption}

\newtheorem{remark}[theorem]{Remark}

 \usepackage{tikz, tikz-3dplot}
 \usepackage{pgfplotstable}
\usepackage{pgfplots}
\usepackage{pgffor,pgfmath}
\usepackage{booktabs}
\usetikzlibrary{arrows}
\pgfplotsset{compat=1.14}

\usepackage{mathtools}

\usepackage[inline,shortlabels]{enumitem}

%% file: commands.tex
\newcommand{\R}{\mathbb{R}}

\newcommand{\Gc}{\mathcal{G}}

\newcommand{\boundary}{B}

\newcommand{\incidenceNB}{\operatorname{K}}

\newcommand{\incidence}{\nabla}

\newcommand{\Ker}{\operatorname{Ker}}

\newcommand{\Dim}{\operatorname{dim}}

\newcommand{\internalnodes}{\nodeset\setminus\boundary}

\newcommand{\Hc}{\mathcal{H}}

\newcommand{\rayl}{\mathcal{R}}

\newcommand{\eigenfunction}[1][]{
  \ifthenelse{\equal{#1}{}}
  {f}
  {f_{#1}}}
\newcommand{\eigenval}[1][]{
  \ifthenelse{\equal{#1}{}}
  {\lambda}
  {\lambda_{#1}}}

\DeclareMathSymbol{\Xdsum}{\mathop}{largesymbols}{88}
\DeclareMathSymbol{\Xtsum}{\mathop}{largesymbols}{80}

\DeclareMathOperator*{\ddsum}{\mathchoice{\Xtsum}{\Xtsum}{\Xtsum}{\Xtsum}}

\newcommand\dDsum{\ddsum\limits}

\newcommand{\diag}{\operatorname{diag}}

\newcommand{\grad}{\nabla}

\newcommand{\plap}{\Delta_p}

\newcommand{\lap}[1][]{
  \ifthenelse{\equal{#1}{}}
  {\Delta}
  {\Delta_{#1}}}

 \newcommand{\Eps}{\mathcal{E}}

\newcommand{\lyap}[1][]{
  \ifthenelse{\equal{#1}{}}
  {\mathcal{L}}
  {\mathcal{L}_{#1}}}

\newcommand{\mass}{\mathrm{M}}

\newcommand{\edgeset}{E}

\newcommand{\nodeset}{V}

\newcommand{\edgelength}{\omega}

\newcommand{\optedgeweight}{\mu^*}

\newcommand{\optnodeweight}{\nu^*}

\newcommand{\edgeweight}[1][]{
  \ifthenelse{\equal{#1}{}}
  {\mu}
  {\mu_{#1}}}
\newcommand{\nodeweight}[1][]{
  \ifthenelse{\equal{#1}{}}
  {\nu}
  {\nu_{#1}}}

\DeclareMathOperator*{\argmax}{arg\:max}
\DeclareMathOperator*{\argmin}{arg\:min}

\newcommand{\morse}[1][]{
  \ifthenelse{\equal{#1}{}}
  {\mathcal{MI}}
  {\mathcal{MI}_{#1}}}

\DeclareMathOperator*{\Span}{span}

\let\oldforall\forall
\renewcommand{\forall}{\oldforall \, }

\let\oldexist\exists
\renewcommand{\exists}{\oldexist \: }